\documentclass[reqno]{amsart} 
\usepackage{amsmath, amsthm, amssymb}
\usepackage[foot]{amsaddr}
\usepackage{bm}
\usepackage{graphicx}

\usepackage{color}
\definecolor{Chocolat}{rgb}{0, 0.15, 0.05}
\definecolor{BleuTresFonce}{rgb}{0, 0.15, 0.05}
\usepackage[utf8]{inputenc}
\usepackage{pifont}

\theoremstyle{plain}
\newtheorem{thm}{Theorem}[section]
\newtheorem{theorem}[thm]{Theorem}
\newtheorem{lemma}[thm]{Lemma}

\newtheorem{proposition}[thm]{Proposition}

\newtheorem*{claim*}{Claim}
\newtheorem*{thm*}{Theorem}

\theoremstyle{definition}

\newtheorem{remark}[thm]{Remark}

\newtheorem{example}[thm]{\sc Example}

 \usepackage{multicol}

\usepackage{float}
\usepackage{stmaryrd}
\usepackage{mathalfa}
\usepackage{euscript}
\usepackage{eufrak}
\usepackage{extarrows}
\usepackage[normalem]{ulem}
 
\numberwithin{equation}{section}

\allowdisplaybreaks
\usepackage{lscape}
\usepackage{tabularx}
\usepackage[dvipsnames]{xcolor}
\usepackage{framed}
\usepackage[title]{appendix}
\usepackage{tikz}
\usetikzlibrary{decorations.markings}
\usetikzlibrary{shapes.geometric,positioning}
\usepackage{booktabs}
\usepackage{euscript}
\usepackage[normalem]{ulem}
\usepackage{array}
\usepackage{collectbox}
\usetikzlibrary{arrows,matrix}

\usepackage{footnote}
\usepackage{makecell}
\pgfdeclarelayer{edgelayer}
\pgfdeclarelayer{nodelayer}
\pgfsetlayers{edgelayer,nodelayer,main}
\usepackage{array}
\usetikzlibrary{decorations.pathmorphing, patterns,shapes, trees}

\usepackage{multirow}

 \usetikzlibrary{calc}
\usepackage{accents}
\definecolor{bluen}{RGB}{0,100,200}

\makeatletter

\pgfarrowsdeclare{open cap}{open cap}
{\pgfarrowsleftextend{+0pt}\pgfarrowsrightextend{+0.5\pgflinewidth}}
{
  \pgfmathsetlength{\pgfutil@tempdimb}{.5*\pgflinewidth-.5*\pgfinnerlinewidth}%
  \pgfsetlinewidth{\pgfutil@tempdimb}
  \pgfsetbuttcap
  \pgfsetdash{}{0pt}
  \pgfmathsetlength{\pgfutil@tempdima}{.5*\pgfutil@tempdimb+.5*\pgfinnerlinewidth}%
  \pgfpathmoveto{\pgfqpoint{0pt}{\pgfutil@tempdima}}
  \pgfpatharc{270}{90}{-\pgfutil@tempdima}
  \pgfusepathqstroke
}
\makeatother

\tikzset{snake it/.style={decorate, decoration=snake}}
\usetikzlibrary{decorations.pathreplacing}
\newcommand\mydots{\hbox to 1em{.\hss.\hss.}}

\newcommand{\ul}[1]{\underline{#1}}
\newcommand{\nin}{\not\in}
\newcommand{\inc}{\subseteq}
\newcommand{\set}[1]{\{#1\}}
\newcommand{\setc}[2]{\set{#1 \mid #2}}
\newcommand{\union}{\cup}		
\newcommand{\inter}{\cap}		
\newcommand{\Union}{\bigcup}		
\newcommand{\restrH}[2]{\hyper{#1}\backslash #2}

\newcommand{\hyper}[1]{{\bf #1}}

\newcommand{\clandec}{\;\uline{\leadsto}\;}
\newcommand{\deldec}{\;\uuline{\leadsto}\;}

\newcommand{\bcdot}{\,{\bm \cdot}\,}

\newcommand{\precstar}{\prec\!\ast}

\newcommand{\astsucc}{\ast\!\succ}

\newcommand{\bcdotass}{\!\bcdot \!{\it ass}}

\newcommand{\succbcdot}{\succ\!\!\bcdot\!}

\newcommand{\bcdotprec}{\!\bcdot\!\!\prec}

 \newcommand{\precbcdotsucc}{\prec\!\!\bcdot\!\!\succ}
 
 \newcommand{\plunderline}[1]{#1}

 \newcommand{\incs}{\subsetneq}
 \newcommand{\pbullet}{\mbox{{\scriptsize $\bullet$}}}
 
 \newcommand{\Friezo}{\hyper{F}}

\newcommand{\restrconstr}[2]{{#1}_{{}^\lceil {#2}}}
\def\mathraise#1#2#3#4{
\mathchoice
{\raisebox{#2}{$\displaystyle{#1}$}}
{\raisebox{#2}{$#1$}}
{\raisebox{#3}{$\scriptstyle{#1}$}}
{\raisebox{#4}{$\scriptscriptstyle{#1}$}}
}

\def\mathraiseord#1#2#3#4{\mathord{\mathraise#1{#2}{#3}{#4}}}
\def\valup{\mathraiseord{\uparrow}{0.29ex}{0.21ex}{0.16ex}}

\usepackage{diagbox}

\definecolor{mauve}{HTML}{9966CC}
\definecolor{rose}{HTML}{DE5D83}
\definecolor{bleu}{HTML}{5AC1C9}

\usepackage{todonotes}

\title{Tridendriform algebras on hypergraph polytopes}

\author{Pierre-Louis Curien $^\dagger$}
\author{Bérénice Delcroix-Oger$^\ast$}
\author{Jovana Obradovi\' c $^\ddagger$} 
\address[$\dagger$]{Université Paris Cité, CNRS, Inria, IRIF, $\pi r^2$ project-team, F-75006, Paris, France}
\address[$\ddagger$]{Mathematical Institute of the Serbian Academy of Sciences and Arts, Belgrade, Serbia}
\address[$\ast$]{Université de Montpellier, CNRS, IMAG, F-34090, Montpellier, France}

\address[]{{curien@irif.fr}, {berenice.delcroix-oger@umontpellier.fr}, {jovana@mi.sanu.ac.rs}. This work was partially supported by ANR projects ANR-20-CE40-0016 HighAGT, ANR-20-CE40-0007 CARPLO, ANR-20-CE48-0010 SSS and ANR-19-CE40-0006 ALCOHOL.}


\usepackage[backend=bibtex]{biblatex}
\begin{document}

\begin{abstract}
{We extend the works of Loday-Ronco and Burgunder-Ronco on the tridendriform decomposition of the shuffle product on the faces of associahedra and permutohedra, to other families of hypergraph polytopes (or nestohedra), including simplices, hypercubes and some new families. We also extend the shuffle product to  take more than two arguments, and define accordingly a new algebraic structure, that we call \emph{polydendriform}, from which the original tridendriform equations can be crisply synthesized.}
\end{abstract}

\maketitle

\tableofcontents
 
 \section{Introduction} \label{Intro-section}
In 1998, Loday-Ronco introduced a Hopf algebra on the linear 
span of rooted planar binary trees \cite{LR-planar-Hopf}. 
This Hopf algebra is closely related to the Malvenuto-
Reutenauer Hopf algebra on permutations \cite{MalvReut}. 
Planar binary trees and permutations label the vertices of 
two well-known families of  polytopes: associahedra and 
permutohedra. The associative products of these Hopf algebras 
were then extended to associative products on all faces of 
these polytopes labeled respectively by planar trees and 
surjections by Loday-Ronco \cite{LR02} and Burgunder-Ronco 
\cite{BR}. 
More precisely, Loday-Ronco introduced an associative product 
$\ast$ on planar trees as a shuffle of trees, where the 
shuffle $T \ast S$ of trees $T$ and $S$ is defined as a 
formal sum of trees whose nodes originate either from $T$, or 
from $S$, or from merging a node of $S$ with a node of $T$. 
Loday and Ronco remarked that it is possible to split this 
product $\ast$  according to where the roots of the resulting trees originate from, giving rise to three operations ``$\prec$'', ``$\succ$'' and ``$\bcdot$'', with $\ast=(\prec) + (\succ) +  (\bcdot)$, forming an algebraic structure called tridendriform. 
For instance, the following product
\begin{align*}
\begin{tikzpicture}[scale=0.3,color=bleu, very thick]
\draw (-1,1)--(0,0)--(2,2);
\draw (1,1)--(-1,3);
\draw (0,2)--(1,3);
\draw[black, very thick] (3.5,1.5) node{$\ast$};
\end{tikzpicture}
\begin{tikzpicture}[scale=0.3,color=rose, very thick]
\draw (-2,2)--(0,0)--(1,1);
\draw (-1,1)--(0,2);
\draw[black] (2,1.5) node{$=$};
\end{tikzpicture}
\begin{tikzpicture}[scale=0.27, very thick]
\draw[bleu] (-1,1)--(0,0)--(2,2);
\draw[bleu] (1,1)--(0,2)--(-0.8,3);
\draw[bleu] (0,2)--(0.8,3);
\draw[rose](0.4,4)--(2,2)--(2.8,3);
\draw[rose](1.2,3)--(2,4);
\end{tikzpicture}
\begin{tikzpicture}[scale=0.3, very thick]
\draw[black] (-2.5,1.5) node{$+$};
\draw[bleu] (-1,1)--(0,0)--(1,1);
\draw[bleu] (1,1)--(-0.5,2)--(-1,3);
\draw[bleu] (-0.5,2)--(0,3);
\draw[rose](0.5,3)--(1,2)--(1.5,3);
\draw[rose](1,1)--(2,2);
\draw[mauve] (1,2)--(1,1);
\end{tikzpicture}
\begin{tikzpicture}[scale=0.27, very thick]
\draw[black] (-2,1.5) node{$+$};
\draw[bleu] (-1,1)--(0,0)--(1,1);
\draw[rose](0,2)--(1,1)--(2,2);
\draw[bleu] (-1,3)--(0,2)--(1,3);
\draw[bleu] (-1.8,4)--(-1,3)--(-0.2,4);
\draw[rose](0.2,4)--(1,3)--(1.8,4);
\end{tikzpicture}
\begin{tikzpicture}[scale=0.3, very thick]
\draw[black] (-2.5,0.5) node{$+$};
\draw[bleu] (0,0)--(1,-1)--(2,0);
\draw[bleu] (1,1)--(-0.5,2)--(-1,3);
\draw[bleu] (-0.5,2)--(0,3);
\draw[rose](1,1)--(2,0)--(3,1);
\draw[rose](1,1)--(2,2);
\draw[mauve] (1,2)--(1,1);
\end{tikzpicture}
\begin{tikzpicture}[scale=0.25, very thick]
\draw[black] (-2,1.5) node{$+$};
\draw[bleu] (-1,1)--(0,0)--(1,1);
\draw[rose](-1,3)--(1,1)--(2,2);
\draw[rose] (0,2)--(1,3);
\draw[bleu] (-2,4)--(-1,3)--(0,4);
\draw[bleu] (-2.8,5)--(-2,4)--(-1.2,5);
\end{tikzpicture}
\\
\begin{tikzpicture}[scale=0.3, very thick]
\draw[black] (-2,1.5) node{$+$};
\draw[bleu] (-1,1)--(0,0);
\draw[mauve] (0,0)--(0,1);
\draw[rose](0,0)--(1,1);
\draw[bleu] (-1,2)--(0,1)--(1,2);
\draw[bleu] (-1.8,3)--(-1,2)--(-0.2,3);
\draw[rose] (0.2,3)--(1,2)--(1.8,3);
\end{tikzpicture}
\begin{tikzpicture}[scale=0.3, very thick]
\draw[black] (-2,1.5) node{$+$};
\draw[bleu] (-1,1)--(0,0);
\draw[mauve] (0,0)--(0,1);
\draw[rose](0,0)--(1,1);
\draw[bleu] (-1,2)--(0,1);
\draw[mauve] (0,1)--(0,2);
\draw[rose](0,1)--(1,2);
\draw[bleu] (-2,3)--(-1,2)--(0,3);
\end{tikzpicture}
\begin{tikzpicture}[scale=0.3, very thick]
\draw[black] (-2,1.5) node{$+$};
\draw[bleu] (-1,1)--(0,0);
\draw[mauve] (0,0)--(0,1);
\draw[rose](0,0)--(1,1);
\draw[rose](-1,2)--(0,1)--(1,2);
\draw[bleu] (-2,3)--(-1,2)--(0,3);
\draw[bleu] (-3,4)--(-2,3)--(-1,4);
\end{tikzpicture}
\begin{tikzpicture}[scale=0.25, very thick]
\draw[black] (-3,1.5) node{$+$};
\draw[rose] (-1,1)--(0,0)--(1,1);
\draw[bleu](-2,2)--(-1,1)--(1,3);
\draw[rose] (0.2,4)--(1,3)--(1.8,4);
\draw[bleu] (0,2)--(-1,3)--(-1.8,4);
\draw[bleu] (-1,3)--(-0.2,4);
\end{tikzpicture}
\begin{tikzpicture}[scale=0.27, very thick]
\draw[black] (-3,1.5) node{$+$};
\draw[rose] (-1,1)--(0,0)--(1,1);
\draw[bleu](-2,2)--(-1,1)--(0,2);
\draw[rose] (0,2)--(1,3);
\draw[mauve] (0,2)--(0,3);
\draw[bleu] (0,2)--(-1,3)--(-1.8,4);
\draw[bleu] (-1,3)--(-0.2,4);
\end{tikzpicture}
\begin{tikzpicture}[scale=0.25, very thick]
\draw[black] (-3,1.5) node{$+$};
\draw[rose] (-1,1)--(0,0)--(1,1);
\draw[bleu](-2,2)--(-1,1)--(0,2);
\draw[rose] (-1,3)--(0,2)--(1,3);
\draw[bleu] (0,4)--(-1,3)--(-2,4)--(-2.8,5);
\draw[bleu] (-1.2,5)--(-2,4);
\end{tikzpicture}
\begin{tikzpicture}[scale=0.27, very thick]
\draw[black] (-3,1.5) node{$+$};
\draw[rose] (-1,1)--(0,0)--(1,1);
\draw[bleu](-2,2)--(-1,1);
\draw[mauve](-1,1)--(-1,2);
\draw[rose] (-1,1)--(0,2);
\draw[bleu] (-2,3)--(-1,2)--(0,3);
\draw[bleu] (-3,4)--(-2,3)--(-1,4);
\end{tikzpicture}
\begin{tikzpicture}[scale=0.25, very thick]
\draw[black] (-2,-0.5) node{$+$};
\draw[rose](0,0)--(2,-2)--(3,-1);
\draw[rose](1,-1)--(2,0);
\draw[bleu] (-1,1)--(0,0)--(2,2);
\draw[bleu] (1,1)--(0,2)--(-0.8,3);
\draw[bleu] (0,2)--(0.8,3);
\end{tikzpicture}
\end{align*}
is split into
\begin{align*}
\begin{tikzpicture}[scale=0.3,color=bleu, very thick]
\draw (-1,1)--(0,0)--(2,2);
\draw (1,1)--(-1,3);
\draw (0,2)--(1,3);
\draw[black, very thick] (3.5,1.5) node{$\prec$};
\end{tikzpicture}
\begin{tikzpicture}[scale=0.3,color=rose, very thick]
\draw (-2,2)--(0,0)--(1,1);
\draw (-1,1)--(0,2);
\draw[black] (2,1.5) node{$=$};
\end{tikzpicture}&
\begin{tikzpicture}[scale=0.27, very thick]
\draw[bleu] (-1,1)--(0,0)--(2,2);
\draw[bleu] (1,1)--(0,2)--(-0.8,3);
\draw[bleu] (0,2)--(0.8,3);
\draw[rose](0.4,4)--(2,2)--(2.8,3);
\draw[rose](1.2,3)--(2,4);
\end{tikzpicture}
\begin{tikzpicture}[scale=0.3, very thick]
\draw[black] (-2.5,1.5) node{+};
\draw[bleu] (-1,1)--(0,0)--(1,1);
\draw[bleu] (1,1)--(-0.5,2)--(-1,3);
\draw[bleu] (-0.5,2)--(0,3);
\draw[rose](0.5,3)--(1,2)--(1.5,3);
\draw[rose](1,1)--(2,2);
\draw[mauve] (1,2)--(1,1);
\end{tikzpicture}
\begin{tikzpicture}[scale=0.27, very thick]
\draw[black] (-2,1.5) node{+};
\draw[bleu] (-1,1)--(0,0)--(1,1);
\draw[rose](0,2)--(1,1)--(2,2);
\draw[bleu] (-1,3)--(0,2)--(1,3);
\draw[bleu] (-1.8,4)--(-1,3)--(-0.2,4);
\draw[rose](0.2,4)--(1,3)--(1.8,4);
\end{tikzpicture}
\begin{tikzpicture}[scale=0.3, very thick]
\draw[black] (-2.5,0.5) node{+};
\draw[bleu] (0,0)--(1,-1)--(2,0);
\draw[bleu] (1,1)--(-0.5,2)--(-1,3);
\draw[bleu] (-0.5,2)--(0,3);
\draw[rose](1,1)--(2,0)--(3,1);
\draw[rose](1,1)--(2,2);
\draw[mauve] (1,2)--(1,1);
\end{tikzpicture}
\begin{tikzpicture}[scale=0.25, very thick]
\draw[black] (-2,1.5) node{+};
\draw[bleu] (-1,1)--(0,0)--(1,1);
\draw[rose](-1,3)--(1,1)--(2,2);
\draw[rose] (0,2)--(1,3);
\draw[bleu] (-2,4)--(-1,3)--(0,4);
\draw[bleu] (-2.8,5)--(-2,4)--(-1.2,5);
\end{tikzpicture} \\
\begin{tikzpicture}[scale=0.3,color=bleu, very thick]
\draw (-1,1)--(0,0)--(2,2);
\draw (1,1)--(-1,3);
\draw (0,2)--(1,3);
\draw[black, very thick] (3.5,1.5) node{$\bcdot$};
\end{tikzpicture}
\begin{tikzpicture}[scale=0.3,color=rose, very thick]
\draw (-2,2)--(0,0)--(1,1);
\draw (-1,1)--(0,2);
\draw[black] (2,1.5) node{=};
\end{tikzpicture}&
\begin{tikzpicture}[scale=0.3, very thick]
\draw[bleu] (-1,1)--(0,0);
\draw[mauve] (0,0)--(0,1);
\draw[rose](0,0)--(1,1);
\draw[bleu] (-1,2)--(0,1)--(1,2);
\draw[bleu] (-1.8,3)--(-1,2)--(-0.2,3);
\draw[rose] (0.2,3)--(1,2)--(1.8,3);
\end{tikzpicture}
\begin{tikzpicture}[scale=0.3, very thick]
\draw[black] (-2,1.5) node{+};
\draw[bleu] (-1,1)--(0,0);
\draw[mauve] (0,0)--(0,1);
\draw[rose](0,0)--(1,1);
\draw[bleu] (-1,2)--(0,1);
\draw[mauve] (0,1)--(0,2);
\draw[rose](0,1)--(1,2);
\draw[bleu] (-2,3)--(-1,2)--(0,3);
\end{tikzpicture}
\begin{tikzpicture}[scale=0.3, very thick]
\draw[black] (-2,1.5) node{+};
\draw[bleu] (-1,1)--(0,0);
\draw[mauve] (0,0)--(0,1);
\draw[rose](0,0)--(1,1);
\draw[rose](-1,2)--(0,1)--(1,2);
\draw[bleu] (-2,3)--(-1,2)--(0,3);
\draw[bleu] (-3,4)--(-2,3)--(-1,4);
\end{tikzpicture} \\
\begin{tikzpicture}[scale=0.3,color=bleu, very thick]
\draw (-1,1)--(0,0)--(2,2);
\draw (1,1)--(-1,3);
\draw (0,2)--(1,3);
\draw[black, very thick] (3.5,1.5) node{$\succ$};
\end{tikzpicture}
\begin{tikzpicture}[scale=0.3,color=rose, very thick]
\draw (-2,2)--(0,0)--(1,1);
\draw (-1,1)--(0,2);
\draw[black] (2,1.5) node{=};
\end{tikzpicture}&
\begin{tikzpicture}[scale=0.25, very thick]
\draw[rose] (-1,1)--(0,0)--(1,1);
\draw[bleu](-2,2)--(-1,1)--(1,3);
\draw[rose] (0.2,4)--(1,3)--(1.8,4);
\draw[bleu] (0,2)--(-1,3)--(-1.8,4);
\draw[bleu] (-1,3)--(-0.2,4);
\end{tikzpicture}
\begin{tikzpicture}[scale=0.27, very thick]
\draw[black] (-3,1.5) node{+};
\draw[rose] (-1,1)--(0,0)--(1,1);
\draw[bleu](-2,2)--(-1,1)--(0,2);
\draw[rose] (0,2)--(1,3);
\draw[mauve] (0,2)--(0,3);
\draw[bleu] (0,2)--(-1,3)--(-1.8,4);
\draw[bleu] (-1,3)--(-0.2,4);
\end{tikzpicture}
\begin{tikzpicture}[scale=0.25, very thick]
\draw[black] (-3,1.5) node{+};
\draw[rose] (-1,1)--(0,0)--(1,1);
\draw[bleu](-2,2)--(-1,1)--(0,2);
\draw[rose] (-1,3)--(0,2)--(1,3);
\draw[bleu] (0,4)--(-1,3)--(-2,4)--(-2.8,5);
\draw[bleu] (-1.2,5)--(-2,4);
\end{tikzpicture}
\begin{tikzpicture}[scale=0.27, very thick]
\draw[black] (-3,1.5) node{+};
\draw[rose] (-1,1)--(0,0)--(1,1);
\draw[bleu](-2,2)--(-1,1);
\draw[mauve](-1,1)--(-1,2);
\draw[rose] (-1,1)--(0,2);
\draw[bleu] (-2,3)--(-1,2)--(0,3);
\draw[bleu] (-3,4)--(-2,3)--(-1,4);
\end{tikzpicture}
\begin{tikzpicture}[scale=0.25, very thick]
\draw[black] (-2,-0.5) node{+};
\draw[rose](0,0)--(2,-2)--(3,-1);
\draw[rose](1,-1)--(2,0);
\draw[bleu] (-1,1)--(0,0)--(2,2);
\draw[bleu] (1,1)--(0,2)--(-0.8,3);
\draw[bleu] (0,2)--(0.8,3);
\end{tikzpicture}
\end{align*}

Burgunder and Ronco  applied a similar  ternary splitting to surjections, also known as packed words, and obtained also a tridendriform structure.

Associahedra and permutohedra are instances of polytopes called \emph{hypergraph polytopes}
 \cite{DP}, which are obtained by truncating some faces of simplices, and are also known as \emph{nestohedra} \cite{P09}. The description of faces of hypergraph polytopes in terms of tree structures -- called constructs -- given in \cite{COI}  provides an adapted framework to extend the setting of Loday-Ronco and Burgunder-Ronco to other families  of polytopes. 

We find it convenient to work in an  ``unbiased'' setting, where our operations may have any finite arity (think of the product $a\times b\times c$ of three numbers $a,b,c$ as opposed to $(a \times b)\times c$ or $a\times (b\times c)$). This leads us to a reformulation of the tridendriform structure (actually $q$-tridendriform  -- see below), that we call polydendriform.  We exhibit conditions under which we can define such a polydendriform structure.
 The underlying (binary) associative product that we obtain coincides with the associative product defined by Ronco \cite{RoncoGTO} on graph associahedra \cite{CD-CCGA}, which are a special type of hypergraph polytopes where the associated hypergraphs have only hyperedges of cardinality two. Our results apply also  to other families of hypergraph polytopes such as simplices, hypercubes and erosohedra.
 
 Therefore, with respect to  \cite{RoncoGTO}, our extension is two-fold:  we describe not only an associative product, but a tridendriform splitting of it, and our framework applies in situations that are not covered by graph associahedra.

The article is organized as follows. 
In  \S\ref{Prologue-section}, we explain in detail the case of the permutohedra,
and motivate and recall Burgunder-Ronco's notion of $q$-tridendriform algebra, i.e.,  an algebra with operations ``$\prec$'', ``$\succ$'' and ``$\bcdot$'',  satisfying the same equations as in tridendriform algebras, but with the associated (associative) product being now defined as $\ast = (\prec) + (\succ) + q( \bcdot)$  for an arbitrary $q \in \mathbb{K}$, where $\mathbb{K}$ is the ambient field.
In \S\ref{reminders-hypergraph-section}, we recall some notions on hypergraph polytopes and constructs.
In \S\ref{strict-teams}, we introduce our conditions for a family of polytopes to have a polydendriform algebra structure. We first define a so-called ``strict'' condition  that makes it possible to define $q$-tridendriform algebras, for arbitrary $q$. We then define a weaker condition called ``semi-strict'', which allows us to deal with a wider class of examples, but for which $q$ has to be $-1$.  In \S\ref{reminders-hypergraph-section} and \S\ref{strict-teams}, we provide a bunch of new examples that do not fit in the framework of graph associahedra, such as friezohedra, simplices, hypercubes and erosohedra.

\section{Prologue} \label{Prologue-section}

We recall Burgunder-Ronco's shuffle product on the faces of permutohedra~\cite{BR}.
We set $[n]=\set{1,\ldots,n}$, and identify a function $f:[n]\rightarrow X$ (for some set $X$) with the sequence $(f(1),\ldots,f(n))$.

\smallskip
By {\em surjection}, we mean a function $f:[m]\rightarrow[n]$ (for some $m,n\geq 1$) that is surjective.  
For arbitrary $h:[m]\rightarrow[n]$, we can build a surjection 
${\tt std}(h):=\phi\circ h: [m]\rightarrow [|{\tt Im}(h)|]$,
where $\phi$ is the unique increasing bijection ${\tt Im}(h)\rightarrow  [|{\tt Im}(h)|]$.
For example, we have
 ${\tt std}(1,4,3,4)=(1,3,2,3)$. Surjections are also known as packed words  \cite{HNT}.
 They label the faces of permutohedra, as shown in \cite{Ch00}.

\smallskip 
If $f:[m_1]\rightarrow[n_1]$ and $g:[m_2]\rightarrow[n_2]$ are surjections, we look for 
 all surjections $(h,k)$ such that
 ${\tt std}(h)=f$ and ${\tt std}(k)=g$. 
Note that we have  then ${\tt Im}(h,k)=[n]$, for some $\max(n_1,n_2)\leq n\leq n_1+n_2$. Below, we do this for 
$f:=(1,2,1)$ and $g:=(2,1)$, underlining the maximum elements  of $h$ and of $k$.

\smallskip 
\small{ $\bullet$ $n=2$:   $(1,\ul{2},1,\ul{2},1)$

\smallskip $\bullet$ $n=3$: $(1,\ul{2},1,\ul{3},1)$, $(1,\ul{3},1,\ul{3},2)$, $(2,\ul{3},2,\ul{2},1)$, $(1,\ul{2},1,\ul{3},2)$,  $(1,\ul{3},1,\ul{2},1)$,  $(2,\ul{3},2,\ul{3},1)$

\smallskip $\bullet$ $n=4$:
$(1,\ul{2},1,\ul{4},3)$, $(1,\ul{3},1,\ul{4},2)$,  $(1,\ul{4},1,\ul{3},2)$,  $(2,\ul{3},2,\ul{4},1)$, $(2,\ul{4},2,\ul{3},1)$, $(3,\ul{4},3,\ul{2},1)$.}
 
\smallskip
\noindent We collect those pairs in the following formal sums (cf. \S\ref{Intro-section}): 
$$\begin{array}{lll}
 f\prec g&  := & {(2,\ul{3},2,\ul{2},1)} 
 +  {(1,\ul{3},1,\ul{2},1)} 
 +  {(1,\ul{4},1,\ul{3},2)} 
 +  {(2,\ul{4},2,\ul{3},1)} 
 +  (3,\ul{4},3,\ul{2},1)  \\
 && \mathit{(\max(h)>\max(k))}\\
  f\bcdot g & := &  {(1,\ul{2},1,\ul{2},1)} +  {(1,\ul{3},1,\ul{3},2)} + { (2,\ul{3},2,\ul{3},1)} \\
  && \mathit{(\max(h)=\max(k))}\\
  {  f\succ g}&  := &  {  (1,\ul{2},1,\ul{3},1)} 
  +  {  (1,\ul{2},1,\ul{3},2)}
  +  {  (1,\ul{2},1,\ul{4},3)}   +  {  (1,\ul{3},1,\ul{4},2)}
    +   {  (2,\ul{3},2,\ul{4},1)} \\
   & &\mathit{(\max(h)<\max(k))}\\
  f\ast g&  := & ({  f\prec g}) + ({  f\bcdot g}) + ({ f\succ g}).
\end{array}$$
\normalsize
   The operations $\prec$, $\bcdot$ and $\succ$ satisfy the following {\em tridendriform} equations
\begin{align*}
&\mathit{(\precstar)}\; (a\prec b)\prec c=a\prec(b\ast c) &
&\mathit{(\succ\prec)}\;(a\succ b)\prec c=a\succ(b\prec c) \\
&\mathit{(\astsucc)} \; (a\ast b)\succ c=a\succ(b\succ c) &
&\mathit{(\bcdotass)} \;(a\bcdot b)\bcdot c=a\bcdot(b\bcdot c) \\
&\mathit{(\succbcdot)} \;(a\succ b)\bcdot c=a\succ(b\bcdot c) &
&\mathit{(\precbcdotsucc)} \; (a\prec b)\bcdot c=a\bcdot (b\succ c)\\
&\mathit{(\bcdotprec)} \; (a\bcdot b)\prec c=a\bcdot (b\prec c), 
&& \text{ and  the operation $\ast$ is associative.}
\end{align*}

\smallskip

The tridendriform structure was first recognized and defined by Loday and Ronco~\cite{LR} on Schr\"oder trees, i.e., planar trees without  unary nodes. We will denote such trees as $\bullet(T_1,\ldots,T_n)$, for $n\neq 1$, where $T_1, \ldots, T_n$ are themselves Schr\"oder trees. The tree with only one leaf is then $\bullet()$.
Schr\"oder trees with at least two leaves label the faces of associahedra.
The three tridendriform operations  (already illustrated in \S\ref{Intro-section}) are defined as follows (with the convention that $\bullet() \ast S=S=S\ast\bullet()$):
$$\begin{array}{lll}
\bullet(S_1,\ldots,S_n)\prec T := \bullet (S_1,\ldots,S_{n-1},S_n\ast T)\\
S\succ \bullet (T_1,\ldots,T_n):= \bullet(S\ast T_1,T_2,\ldots,T_n) \\
\bullet(S_1,\ldots,S_m) \:\bcdot\: \bullet(T_1,\ldots,T_n) \; :=\; \bullet(S_1,\ldots,S_{m-1}, S_m\ast T_1,T_2\ldots,T_n) .&&
\end{array}$$

Associahedra and permutohedra are examples of hypergraph polytopes, also known as nestohedra \cite{P09,FS05}. Our goal is to define in this more general framework, and under suitable conditions, an associative product, with associated tridendriform decomposition, instantiating to these two examples and  more.

\smallskip

We close this section by studying the relation between tridendriform structures and associativity more closely.  Burgunder and Ronco~\cite{BR} have introduced a variation of tridendriform algebras, called $q$-tridendriform  algebras (for $q\in \mathbb{R}$, or more generally $q\in {\Bbbk}$ for some field ${\Bbbk}$), where the equations are the same as above, except that now the operation $\bcdot$ is weighted, i.e., $a\ast b$ is redefined as $(a\prec b)+q(a\bcdot b)+(a\succ b)$. This is justified by the following proposition.
\begin{proposition} \label{tridend-implies-ass}
Setting $a\ast b:=\lambda_1(a\prec b)+\lambda_2(a\bcdot b)+\lambda_3(a\succ b)$, if the tridendriform equations are satisfied (with this definition of $\ast$), then $\ast$ is associative if  $\lambda_1=\lambda_3=1$.
\end{proposition} 
\begin{proof} We match
 $$\begin{array}{ll}
 \lambda_1\lambda_1 \underbrace{ (a\prec b)\prec c}_{(\precstar)} \:+\: \lambda_1\lambda_2\underbrace{(a\bcdot b)\prec c}_{(\bcdotprec)} \:+\: \lambda_1\lambda_3 
\underbrace{(a\succ b)\prec c}_{(\succ\prec)}\\
  \:+\: \lambda_2\lambda_1 \underbrace{(a\prec b)\bcdot c}_{(\precbcdotsucc)} \:+\: \lambda_2\lambda_2 \underbrace{(a\bcdot b)\bcdot c}_{(\bcdotass)}
\:+\: \lambda_2\lambda_3 \underbrace{(a\succ b)\bcdot c}_{(\succbcdot)}\\
 \:+\: \lambda_3\underbrace{ (a\ast b)\succ c}_{(\astsucc)}
 \end{array}$$
 with
 $$\begin{array}{ll}
 \lambda_1 \underbrace{a\prec (b\ast c)}_{(\precstar)} \\
 \:+\:
\lambda_2 \lambda_1 \underbrace{ a\bcdot (b\prec c)}_{(\bcdotprec)} \:+\: \lambda_2\lambda_2 \underbrace{a\bcdot (b\bcdot c)}_{(\bcdotass)}\:+\:
\lambda_2\underbrace{\lambda_3 a\bcdot (b\succ c)}_{(\precbcdotsucc)} \\
 \:+\: \lambda_3\lambda_1 \underbrace{a\succ(b\prec c)}_{(\succ\prec)} \:+\: \lambda_3\lambda_2 \underbrace{a\succ(b\bcdot c)}_{(\succbcdot)} \:+\:
\lambda_3\lambda_3 \underbrace{a\succ(b\succ c)}_{(\astsucc)}
\end{array}$$
 using $(\precstar)$ (resp. $(\astsucc)$, $(\precbcdotsucc)$) and the assumption $\lambda_1=1$ (resp.  $\lambda_3=1$, $\lambda_1=\lambda_3$).
\end{proof}
\section{Hypergraph polytopes} \label{reminders-hypergraph-section}

A hypergraph is given by a set  $H$ of vertices (the carrier), and a subset 
$\hyper{H}\inc {\mathcal{P}}(H)\backslash\emptyset$ such that $\Union \hyper{H}=H$.
 The elements of $\hyper{H}$ are called the {\em hyperedges} of $\hyper{H}$.  
 We always assume that $\hyper{H}$ is {\em atomic}, by which we mean that 
 $\set{x}\in \hyper{H}$, for all $x\in H$. 
 Identifying $x$ with $\set{x}$, $H$ can be seen as the set of  hyperedges of 
 cardinality $1$, also called {\em vertices}. We shall use the convention to 
 give the same name to the hypergraph and to its carrier, the former being 
 the bold version of the latter. 
A hyperedge of cardinality 2 is called an {\em edge}.  Note that any ordinary graph $(V,E)$ can be viewed as the atomic hypergraph
$\setc{\set{v}}{v\in V} \union \setc{e}{e\in E}$ (with no hyperedges of cardinality $\geq 3$). 

\smallskip
 
If $\hyper{H}$ is a hypergraph,  and if  $X\inc H$, we set
$\hyper{H}_X:=\setc{Z}{Z\in \hyper{H}\;\mbox{and}\; Z\inc X}$, and $\restrH{H}{X}=\hyper{H}_{H\backslash X}$.
We say that $\hyper{H}$ is {\em connected} if there is no non-trivial partition $H=X_1\union X_2$ such that $\hyper{H}=\hyper{H}_{X_1}\union \hyper{H}_{X_2}$, and that $X\inc H$ is connected in $\hyper{H}$ if $\hyper{H}_X$ is connected.
For each finite hypergraph there exists a partition
$H=X_1\union\ldots\union X_m$ such that each $\hyper{H}_{X_i}$ is connected and $\hyper{H}=\Union(\hyper{H}_{X_i})$.  The $\hyper{H}_{X_i}$ are  the {\em connected components} of $\hyper{H}$. The notation
$\hyper{H},X  \leadsto \hyper{H}_1,\ldots, \hyper{H}_n$
 will mean that  $\hyper{H}_1,\ldots,\hyper{H}_n$ are  the
 connected components of $\restrH{H}{X}$.  

\smallskip

Do\v sen and Petri\'c~\cite{DP} have proposed the following insightful reading of the data of a finite connected hypergraph $\hyper{H}$ as a truncated simplex: the elements of $H$ are identified with the facets (i.e. codimension 1 faces) of the $(|H|-1)$-dimensional simplex, and each $\emptyset\incs X\incs H$, $|X|\geq 2$, such that    $\hyper{H}_X$ is connected designates the intersection of the facets in $X$ as a face to be truncated.
The obtained polytopes, called \emph{hypergraph polytopes}, extend the construction of graph associahedra \cite{CD-CCGA, Zel06}, and are equivalent to nestohedra, introduced by Postnikov \cite{P09}.  
Moreover,  the faces of the   polytope obtained by performing  all the prescribed truncations  are labeled by non-planar trees whose nodes are decorated by non-empty subsets of $H$, called {\em constructs}\footnote{\label{construct-tubing}Constructs as presented here are  just an alternative description of the tubings and of the nested sets  in the literature on graph associahedra and nestohedra, respectively. For a given construct $T$, each tube of the associated tubing is given by a node of $T$ and all its descendants. There are  as many tubes in the tubing as nodes in the construct.}
, whose recursive definition  is given next using a syntax introduced in \cite{COI}:

\smallskip
Let  $\emptyset\neq Y\subseteq H$. If   $\hyper{H},Y  \leadsto \hyper{H}_1,\ldots, \hyper{H}_n$, and if  $T_1,\ldots,T_n$ are constructs of $\hyper{H}_1,\ldots,\hyper{H}_n$, respectively, then the tree obtained by grafting $T_1,\ldots,T_n$ on the root node decorated by $Y$, denoted by $Y(T_1,\ldots,T_n)$ (or sometimes $Y\setc{T_i}{1\leq i\leq n}$), is a construct of  $\hyper{H}$. We write $Y=\mbox{root}(Y(T_1,\ldots,T_n))$. 
 The base case is when $Y=H$ (and hence $n=0$): then the one-node tree $H()$ (written simply $H$) is a construct.
We write $T:\hyper{H}$ to denote that $T$ is a construct of $\hyper{H}$.

\smallskip
This description of faces as trees is particularly nice for encoding face inclusions: by contracting an edge of a construct representing a face of dimension $p$, and  merging the decorations of the two nodes related by that edge, one gets a face of dimension $p+1$, as illustrated below. We shall not make use of this partial order on faces, but it helps in understanding what is going on in the pictures.
\begin{center}
\resizebox{7.75cm}{!}{\begin{tikzpicture}[scale=0.7]
    \node (Y) [rectangle,draw=none,minimum size=0.75cm,inner sep=0.0mm] at (0,-0.9) {\Large $Y$};
    \node (X) [rectangle,draw=none,minimum size=0.75cm,inner sep=0.1mm]  at (-1.5,1) {\Large $X$};
    \node (T11) [circle,draw=none,minimum size=0.75cm,inner sep=0.1mm] at (-2.5,2.5) {\large  $T_{11}$};
   \node (d1) [circle,draw=none,minimum size=0.75cm,inner sep=0.1mm] at (-1.5,2.5) { $\cdots$};
   \node (T1m) [circle,draw=none,minimum size=0.75cm,inner sep=0.1mm] at (-0.5,2.5) {\large $T_{1m}$};
   \node (T2) [circle,draw=none,minimum size=0.75cm,inner sep=0.1mm] at (1,2.5) {\large $T_{2}$};
   \node (d2) [circle,draw=none,minimum size=0.75cm,inner sep=0.1mm] at (2,2.5) {  $\cdots$};
   \node (Tn) [circle,draw=none,minimum size=0.75cm,inner sep=0.1mm] at (3,2.5) {\large  $T_{n}$};
 
  \node (E) [circle,draw=none,minimum size=0.75cm,inner sep=0.1mm] at (4.75,0.5) {\Large $\subseteq$};

   \node (Y1) [rectangle,draw=none,minimum size=0.75cm,inner sep=1.5mm] at (9,-0.9) {\Large  $Y\cup X$};
    \node (T111) [circle,draw=none,minimum size=0.75cm,inner sep=0.1mm] at (6,2.5) {\large  $T_{11}$};
   \node (d11) [circle,draw=none,minimum size=0.75cm,inner sep=0.1mm] at (7.25,2.5) { $\cdots$};
   \node (T1m1) [circle,draw=none,minimum size=0.75cm,inner sep=0.1mm] at (8.5,2.5) {\large  $T_{1m}$};
   \node (T21) [circle,draw=none,minimum size=0.75cm,inner sep=0.1mm] at (10,2.5) {\large  $T_{2}$};
   \node (d21) [circle,draw=none,minimum size=0.75cm,inner sep=0.1mm] at (11,2.5) {\large  $\cdots$};
   \node (Tn1) [circle,draw=none,minimum size=0.75cm,inner sep=0.1mm] at (12,2.5) { \large $T_{n}$};
\draw (T111)--(Y1);
\draw (T1m1)--(Y1);
\draw (Y1)--(T21);
\draw (Y1)--(Tn1);
\draw[very thick,red] (Y)--(X);
\draw (T11)--(X);
\draw (T1m)--(X);
\draw (Y)--(T2);
\draw (Y)--(Tn);
   \end{tikzpicture}
 }
\end{center}

We next give examples of hypergraph polytopes, most of which will be revisited later in the paper.
\begin{example}
{\bf Simplices} are ``encoded'' as the hypergraphs  $$\hyper{S}^X=\setc{\set{x}}{x\in X}\union\set{\set{X}}$$ (no truncation prescribed). The constructs have the form $Y(\set{y_1}\ldots,\set{y_k})$ where $\emptyset\incs Y\inc X$ and $\{y_1, \ldots, y_k\}=H\backslash Y$, pictured as \begin{tikzpicture}[scale=0.5, grow=up]\node{Y}
   child{node{$y_k$}}
   child{node{$\ldots$}}
   child{node{$y_2$}}
   child{node{$y_1$}};
\end{tikzpicture},
and are therefore in bijection with the non-empty subsets of $X$, which can also be seen as pairs $(X,Y)$ standing for $X$ in which all elements of $Y$ have been pointed.

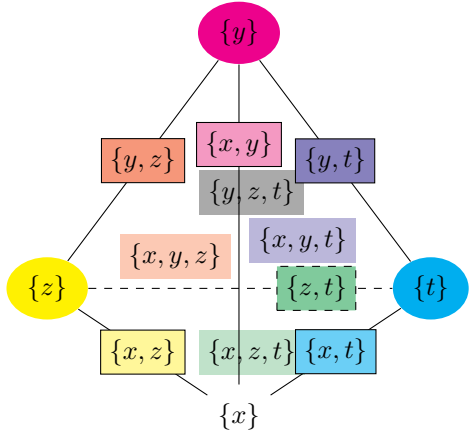
\begin{figure}

\begin{tikzpicture}[scale=0.85]
\node[fill=White, ellipse] (A) at (0,-1){$\{{x}\}$};
\node[fill=Magenta, , ellipse] (B) at (0,5){$\{{y}\}$};
\node[fill=Yellow, ellipse] (C) at (-3,1){$\{{z}\}$};
\node[fill=Cyan, ellipse] (D) at (3,1){$\{{t}\}$};
\node[fill=Cyan!50!Yellow!50] (ACD) at (0.2,0){$\{{x,z,t}\}$};
\node[fill={rgb:Magenta,255;Cyan,255;Yellow,255}] (BCD) at (0.2,2.5){$\{{y,z,t}\}$};
\draw[dashed] (C)--(D) node [near end, draw, fill=Yellow!50!Cyan] {$\{{z,t}\}$};
\draw (A)--(C) node [midway, draw, fill=Yellow!50] {$\{x,z\}$};
\draw (A)--(D) node [midway, draw, fill=Cyan!50] {$\{{x,t}\}$};
\draw (C)--(B) node [midway, draw, fill=Magenta!50!Yellow] {$\{{y,z}\}$};
\draw (D)--(B) node [midway, draw, fill=Cyan!50!Magenta] {$\{{y,t}\}$};
\node[fill=Magenta!50!Yellow!50] (ABC) at (-1,1.5){$\{{x,y,z}\}$};
\node[fill=Magenta!50!Cyan!50] (ABD) at (1,1.75){$\{{x,y,t}\}$};
\draw (A)--(B) node [near end, draw, fill=Magenta!50] {$\{{x,y}\}$};
\end{tikzpicture}
\caption{Simplex $\hyper{S}^{\{x,y,z,t\}}$ labeled by the  subsets associated to each face. Faces of dimension $k$ are labeled by  sets of size $k+1$. 
We blended colors to represent inclusion of faces. The interior of the simplex is labeled by $\{{x,y,z,t}\}$.}
\end{figure} 
\end{example}
\begin{example}
In order to illustrate  how the hypergraph structure dictates truncations, consider the hypergraph 
$$\hyper{C}=\set{\set{x},\set{y},\set{z},\set{y,z}, \set{x,y,z}},$$
obtained from $\hyper{S}^{\{x,y,z\}}$ by adding the edge $\{y,z\}$. The construct 
 $\set{x}(\set{y},\set{z}):\hyper{S}^{\{x,y,z\}}$ is {\em not} a construct of $\hyper{C}$, since $\set{y,z}$ is connected in $\hyper{C}$.  Instead, $\hyper{C}$ features 3 new constructs: 
$\set{x}(\set{y}(\set{z}))$, $\set{x}(\set{z}(\set{y}))$ and $\set{x}(\set{y,z})$, encoding two  vertices and one edge, obtained by truncating  the vertex $\set{x}(\set{y},\set{z})$ of $\hyper{S}^{\{x,y,z\}}$.
We illustrate   this example in Figure \ref{figHyp}. In this figure, and elsewhere, we allow ourselves to write, say $xy$ for $\set{x,y}$. As we shall see, $\hyper{C}$ encodes the hypercube of dimension 2.

\begin{figure}
  \begin{tikzpicture}[thick]
\coordinate (x) at (0:1);
\coordinate (y) at (120:1);
\coordinate (z) at (240:1) ;
\fill[Magenta,opacity=0.5] (x)--(y)--(z)--(x);
\draw[line width=2pt] (y)--(z);
\draw (x) node[draw,ellipse, fill=white]{$x$};
\draw (y) node[draw,ellipse, fill=white]{$y$};
\draw (z) node[draw,ellipse, fill=white]{$z$};
\end{tikzpicture}
  \begin{tikzpicture}[thick]
\coordinate (x) at (0:1);
\coordinate (y) at (120:1);
\coordinate (z) at (240:1) ;
\draw[dotted] (y) --(x) node[midway, above]{$x$};
\draw[dotted] (y) --(z) node[midway, left]{$y$};
\draw[dotted] (z) --(x) node[midway, below]{$z$};
\draw [thick,cyan, fill=cyan,opacity=0.5]
      (barycentric cs:y=1,z=2) --
      (barycentric cs:x=1,z=2) --
      (x)--(y)--cycle;
\end{tikzpicture}
\begin{tikzpicture}[scale=0.5, grow=up]
\node{$x$}
   child{node{$y$}
		child{node{$z$}
		}   
   }
;
\end{tikzpicture}
\begin{tikzpicture}[scale=0.5, grow=up]
\node{$x$}
   child{node{$z$}
		child{node{$y$}
		}   
   }
;
\end{tikzpicture}
\begin{tikzpicture}[scale=0.5, grow=up]
\node{$y$}
   child{node{$x$}
   }
		child{node{$z$}
		}   
;
\end{tikzpicture}
\begin{tikzpicture}[scale=0.5, grow=up]
\node{$z$}
   child{node{$x$}
   }
		child{node{$y$}
		}   
;
\end{tikzpicture}
\begin{tikzpicture}[scale=0.5, grow=up]
\node{$x$}
   child{node{$yz$}
   }
;
\end{tikzpicture}
\begin{tikzpicture}[scale=0.5, grow=up]
\node{$xy$}
   child{node{$z$}
   }
;
\end{tikzpicture}
\begin{tikzpicture}[scale=0.5, grow=up]
\node{$xz$}
   child{node{$y$}
   }
;
\end{tikzpicture}
\begin{tikzpicture}[scale=0.5, grow=up]
\node{$yz$}
   child{node{$x$}
   }
;
\end{tikzpicture}
\begin{tikzpicture}[scale=0.5, grow=up]
\node{$xyz$};
\end{tikzpicture}
\caption{The hypergraph, the truncated simplex and the constructs associated with the hypercube $\hyper{C}$} \label{figHyp}
\end{figure}
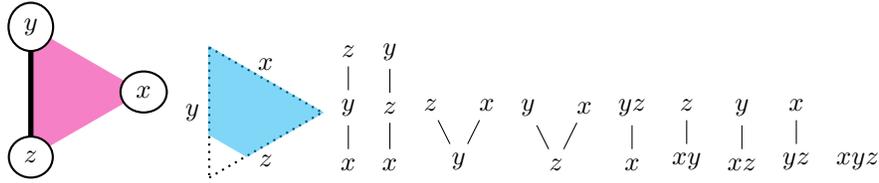
\end{example}
\begin{example}
As a slightly more involved example, we show in Figure \ref{hemiassoc} the polytope encoded by the hypergraph 
$${\bf H}=\{\{x\},\{y\},\{u\},\{v\},\{x,y\},\{x,u\},\{x,v\},\{u,v\},\{x,y,u,v\}\},$$ obtained from the  tetrahedron by truncating three of its vertices and four of its edges. We also ``zoom in'' into the square corresponding to the truncation prescribed by $\{u,v\}$  and label its four 1-dimensional and four 0-dimensional faces by the appropriate constructs of ${\bf H}$. 

\begin{figure}
\centering
  \resizebox{4cm}{!}{
\begin{tikzpicture}[thick,scale=2]
\coordinate (A1) at (0,2);
\coordinate (A11) at (-0.39,1.5);
\coordinate (A111) at (-0.25,1.498);
\coordinate (A112) at (-0.33,1.46);
\coordinate (A12) at (0.39,1.5);
\coordinate (A121) at (0.25,1.498);
\coordinate (A122) at (0.33,1.46);
\coordinate (A13) at (0,1.25);
\coordinate (A131) at (-0.153,1.35); 
\coordinate (A132) at (0.153,1.35); 
\coordinate (A2) at (0,0); 
\coordinate (A21) at (-0.387,0.213); 
\coordinate (A211) at (-0.29,0.2795); 
\coordinate (A212) at (-0.28,0.213); 
\coordinate (A22) at (0.387,0.213); 
\coordinate (A23) at (0,0.5); 
\coordinate (A231) at (-0.153,0.388); 
\coordinate (A232) at (0.153,0.388); 
\coordinate (A3) at (-1.1,0.6);
\coordinate (A31) at (-0.9,0.49);
\coordinate (A311) at (-0.88,0.59);
\coordinate (A312) at (-0.84,0.51);
\coordinate (A32) at (-0.8,0.976);
\coordinate (A321) at (-0.825,0.845);
\coordinate (A322) at (-0.745,0.875);
\coordinate (A33) at (-0.6,0.6);
\coordinate (A4) at (1.1,0.6);
\coordinate (A41) at (1.027,0.565);
\coordinate (A42) at (0.95,0.6);
\draw[draw=ForestGreen]  (A3)--(A2);
\draw[draw=ForestGreen]  (A1)--(A12);
\draw[draw=ForestGreen] (A2)--(A22);
\draw[draw=ForestGreen] (A2) -- (A1);
\draw[draw=ForestGreen] (A3)--(A1);
\draw[draw=ForestGreen,dashed]  (A33) -- (A3);
\draw[draw=ForestGreen,dashed]  (A42) -- (A4);
\draw[draw=ForestGreen] (A41)--(A4)--(A12);
 \draw[draw=black,fill=none]   (A321)--(A311);
 \draw[draw=black,fill=none] (A212)-- (A22) -- (A41);
\draw (A311)--(A312);
\draw (A111)--(A121);
\draw (A111)--(A112);
\draw (A311)--(A211)--(A212) --(A312);
\draw (A112)--(A321);
\draw[dashed] (A321)--(A322)--(A111);
\draw  (A211)--(A231)--(A232)--(A22);
\draw[dashed]  (A33) -- (A42);
\draw (A231) -- (A131) -- (A132) -- (A232) -- cycle;
\draw[dashed] (A322) -- (A33) -- (A312);
\draw (A112) -- (A131) --(A132)--(A122);
\fill[cyan] (A41) -- (A42) -- (A121)--(A122)--cycle;
\draw[dashed]  (A41) -- (A42) -- (A121); 
\draw (A41)--(A122) --(A121);
\end{tikzpicture}} \quad\quad  
\raisebox{1em}{\begin{tikzpicture}[thick]
\coordinate (S1) at (-0.15,0);
\coordinate (S2) at (2.15,0);
\coordinate (S3) at (2.15,1);
\coordinate (S4) at (-0.15,1);
\draw[fill=cyan,opacity=0.3] (S1)--(S2)--(S3)--(S4)-- cycle;
\draw (S1)--(S2)--(S3)--(S4)-- cycle;
\node (s1) at (-1.55,-0.3) {\scriptsize $\{x\}(\{y\},\{u\}(\{v\}))$};
\node (s2) at (3.55,-0.3) {\scriptsize $\{x\}(\{y\},\{v\}(\{u\}))$};
\node (s4) at (-1.55,1.3) {\scriptsize $\{y\}(\{x\}(\{u\}(\{v\})))$};
\node (s3) at (3.55,1.3) {\scriptsize $\{y\}(\{x\}(\{v\}(\{u\})))$};
\node (s14) at (-1.35,0.5) {\scriptsize $\{x,y\}(\{u\}(\{v\}))$};
\node (s23) at (3.35,0.5) {\scriptsize $\{x,y\}(\{v\}(\{u\}))$};
\node (s12) at (1,-0.3) {\scriptsize $\{x\}(\{y\},\{u,v\})$};
\node (s34) at (1,1.3) {\scriptsize $\{y\}(\{x\},\{u,v\})$};
\node (s) at (1,0.5) {\scriptsize $\{x,y\}(\{u,v\})$};

\end{tikzpicture}}
\caption{   A truncated simplex \label{hemiassoc}}
\end{figure}
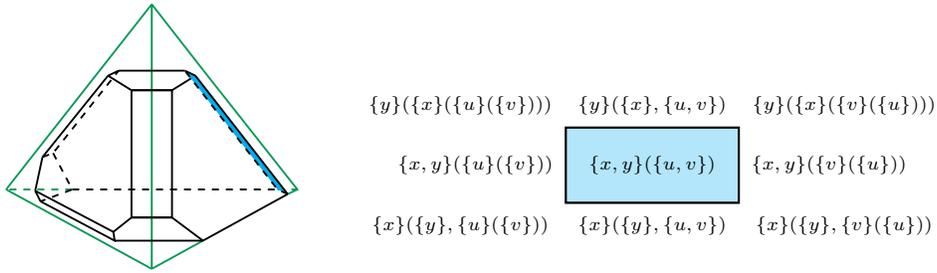
\end{example}

We now give two examples that  do not fit in the framework of graph associahedra: hypercubes and erosohedra.
 
\begin{example} {\bf Hypercubes}. \label{hypercube1-example}
For a finite ordered set $X=\set{x_1<\cdots<x_n}$, consider the hypergraph 
 
 $$\hyper{C}^X=\setc{\setc{x_j}{1\leq j\leq i}}{1\leq i\leq n}.$$
 
 \noindent
The constructs of $\hyper{C}^X$ are in one-to-one correspondence with the set of words of length $n$ over the
alphabet $\set{+,-,\pbullet}$ starting with $+$, and hence decorate the faces of an $(n-1)$-dimensional hypercube. More precisely, we recursively read a construct from such a word $(w_1+w_2)$, where $+$ does not occur in $w_2$,  as follows:
\begin{itemize} 
\item[-] The positions of the occurrences of $\pbullet$ in $w_2$  plus the last occurrence of $+$ in $w$, form the root $R$ of the construct. 
\item[-]  $w_1$ encodes a construct $S$ (if not empty). 
\item[-] The children of $R$ in the  construct are $S$ (if any), and the positions of  the occurrences of $-$ in $w_2$.
\end{itemize}
For instance, the constructs \begin{center}\resizebox{!}{1cm}
{\begin{tikzpicture}[scale=0.5, grow=up]
\node{$3$}
   child{node{$4$}
   }
   child{node{$1\ 2$}
   }
;
\end{tikzpicture}},
\resizebox{!}{1cm}
{\begin{tikzpicture}[scale=0.5, grow=up]
\node{$2 \ 3$}
   child{node{$4$}
   }
   child{node{$1$}
   }
;
\end{tikzpicture}},
\resizebox{!}{1.6cm}
{\begin{tikzpicture}[scale=0.5, grow=up]
\node{$4$}
   child{node{$1$}
          child{node{$3$}   }
          child{node{$2$}   }
   }
;
\end{tikzpicture}}
and 
\resizebox{!}{2.2cm}
{\begin{tikzpicture}[scale=0.5, grow=up]
\node{$4$}
   child{node{$3$}
          child{node{$2$}   
          child{node{$1$}   }}         
   }
;
\end{tikzpicture}}

\end{center}
of $\hyper{C}^{\{1, \ldots, 4\}}$ correspond to the words
\begin{center}
$+\bullet+-$  \quad\quad $++\bullet-$  \quad\quad $+--+$  \quad\quad and  \quad\quad $++++$\, ,
\end{center}
respectively. 
We have already listed all the constructs of the hypercube $\hyper{C}^{\{y<z<x\}}$ in Figure \ref{figHyp}.
The corresponding words are, in this order:
$$(+\!-\!+)\;\;(+\!+\!+)\;\;(+\!-\!-)\;\;(+\!+\!-)\;\;(+\,\pbullet\,+)\;\;(+\!-\pbullet)\;\;(+\!+\pbullet)\;\;(+\,\pbullet\,-)\;\;(+\,\pbullet\,\pbullet).$$
\end{example}
\begin{example}  {\bf Erosohedra}. \label{eroso} They are obtained by cutting every vertex in the simplex. We name them so by  analogy with erosion of rocks. Erosohedra in dimension 2 and 3 are represented on Figure \ref{fig:erosohedra}. The associated hypergraphs are given by:
\begin{align*}
 \hyper{E}^X=\setc{\setc{x_j}{j \neq i}}{1\leq i\leq n},
\end{align*} 
where $X=\{x_1, \ldots, x_n\}$.

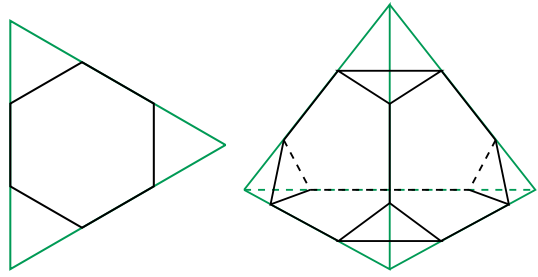
\begin{figure}
  \resizebox{3cm}{!}{
\begin{tikzpicture}[thick,scale=2]
\coordinate (A1) at (0:1);
\coordinate (A2) at (120:1);
\coordinate (A3) at (240:1);

\draw[draw=ForestGreen]  (A1)--(A2)--(A3)--(A1);

\draw [thick,black]
      (barycentric cs:A1=2,A2=1,A3=0) --
      (barycentric cs:A1=1,A2=2,A3=0) --
      (barycentric cs:A1=0,A2=2,A3=1) --
            (barycentric cs:A1=0,A2=1,A3=2) --
            (barycentric cs:A1=1,A2=0,A3=2) --
            (barycentric cs:A1=2,A2=0,A3=1) --
       cycle;
\end{tikzpicture}}
  \resizebox{4cm}{!}{
\begin{tikzpicture}[thick,scale=2]
\coordinate (A1) at (0,2);
\coordinate (A11) at (-0.39,1.5);
\coordinate (A12) at (0.39,1.5);
\coordinate (A13) at (0,1.25);
\coordinate (A2) at (0,0); 
\coordinate (A21) at (-0.387,0.213); 
\coordinate (A22) at (0.387,0.213); 
\coordinate (A23) at (0,0.5); 
\coordinate (A3) at (-1.1,0.6);
\coordinate (A31) at (-0.9,0.49);
\coordinate (A32) at (-0.8,0.976);
\coordinate (A33) at (-0.6,0.6);
\coordinate (A4) at (1.1,0.6);
\coordinate (A41) at (0.9,0.49);
\coordinate (A42) at (0.8,0.976);
\coordinate (A43) at (0.6,0.6);

\draw[draw=ForestGreen]  (A3)--(A2)--(A1)--(A3);
\draw[draw=ForestGreen, dashed] (A3)--(A4);
\draw[draw=ForestGreen] (A2)--(A4)--(A1);
 \draw[thick,draw=black,fill=none]   (A33)--(A31)--(A32);
 \draw[dashed, thick] (A32)--(A33);
 \draw[thick](A32)--(A11);
 \draw[thick,draw=black,fill=none]   (A13)--(A11)--(A12)--cycle;
 \draw[thick,draw=black,fill=none]   (A23)--(A21)--(A22)--cycle;
 \draw[thick](A23)--(A13);
\draw[thick](A42)--(A12);
 \draw[thick,draw=black,fill=none]   (A43)--(A41)--(A42);
 \draw[dashed, thick] (A42)--(A43);
 \draw[thick] (A31)--(A21);
 \draw[thick] (A41)--(A22);
 \draw[dashed, thick] (A33)--(A43);
\end{tikzpicture}}
\caption{Erosohedra in dimension $2$ and $3$}\label{fig:erosohedra}
\end{figure}
The constructs of the erosohedra are of the form
\begin{center}
\begin{tikzpicture}[scale=0.5, grow=up]
\node{$Y$}
   child{node{$y_k$}
   }
   child{node{$\ldots$}
   }
   child{node{$y_1$}
   }
;
\end{tikzpicture}
\quad (with  $|Y|\geq 2$) \quad or \quad
 \begin{tikzpicture}[scale=0.5, grow=up]
\node{$x$}
   child{node{$Z$}
   child{node{$y_k$}
   }
   child{node{$\ldots$}
   }
   child{node{$y_1$}
   }
   }
;
\end{tikzpicture}
\end{center}  
\end{example}
 The number of faces in the erosohedron is given as follows.
\begin{lemma}
The number of vertices in the erosohedron of dimension $n$ is  $n(n+1)$ and the number of faces of dimension $k$ is  $(n-k) \binom{n}{k+1}$. The total number of faces is thus  $2^{n-1}(n+2)-2n-1$ for $n>1$.
\end{lemma}
\begin{proof}
The vertices of the erosohedron correspond to constructs of the form
$$\{x\}(\{y\}( \{z_1\}, \ldots, \{z_k\})),$$ where $x$ and $y$ are two vertices in $X$ and $\{z_1, \ldots, z_k\}=X\backslash\{x,y\}$. 
Faces of dimension $k>0$  correspond to two types of constructs: \\
\noindent
$\;\bullet\;$ $\{x_0, x_1, \ldots x_{k}\} (\{y_1\}, \ldots, \{y_p\})$, where $\{y_1, \ldots, y_p\}=X\backslash\{x_0,\ldots, x_k\}$,\\
\noindent
$\;\bullet\;$ and $\{x_0\}(\{x_1, \ldots x_{k+1}\}(\{y_1\}, \ldots, \{y_p\}))$, where $\{y_1, \ldots, y_p\}=X\backslash\{x_0,\ldots, x_{k+1}\}$. \\
\noindent
Hence, there are $(n-k-1)\binom{n}{k+1} + \binom{n}{k+1}$ such faces.
The total number of faces is then given by summing the previous formulas.
\end{proof}
\begin{example}
We  get {\bf associahedra} and {\bf permutohedra} from the  linear  and complete graphs, respectively: 
$$
\hyper{K}^X= \set{\set{x_1},\ldots,\set{x_n},\set{x_1,x_2},\ldots,\set{x_{n-1},x_n},\set{x_1,\ldots, x_n}},$$
for  $X=\set{x_1<\cdots<x_n}$, and
$$\hyper{P}^X= \set{\set{x_1},\ldots,\set{x_n},\set{x_1,\ldots x_n}} \cup \setc{\set{x_i,x_j}}{1\leq i\neq j\leq n}$$
for $X=\set{x_1,\ldots,x_n}$. One can indeed check that the constructs of $\hyper{K}^X$ (resp. $\hyper{P}^X$) are in one-to-one correspondence with  the planar trees (resp.  surjections) of \S\ref{Prologue-section}. The labeling enabling to identify planar trees with constructs of the associahedra is obtained as a generalization of the one for binary search trees: given a planar tree with root of arity $p+1$, the root is labeled by $\set{x_{i_1}, \ldots, x_{i_p}}$, and each subtree $T_0, \ldots, T_p$ is labeled recursively, in such a way that the condition $\max T_j <x_{i_{j+1}}< \min T_{j+1}$ for any $0 \leq j \leq p-1$ is satisfied.
For permutations, note that we can arrange the data of a surjection $f:[m]\rightarrow[n]$ as the linear construct $f^{-1}(1)(f^{-1}(2)(\ldots  (f^{-1}(n))))$ of height $n$.
\end{example}

\begin{figure}
\centering
  \begin{tikzpicture}[thick]
\node (S1) at (-1,1){2};
\node (S2) at (0,1){4};
\node (S3) at (1,1) {6};
\node (S4) at (2,1) {8};
\draw (S1)--(S2)--(S3)--(S4);
\end{tikzpicture} \hspace{10pt}
\begin{tikzpicture}[thick]
\node (S1) at (0,1){1};
\node (S2) at (0,0){2};
\node (S3) at (1,1) {3};
\node (S4) at (2,1) {5};
\draw (S3)--(S1)--(S2)--(S3)--(S4);
\end{tikzpicture} \hspace{10pt}
\begin{tikzpicture}[thick]
\node (S1) at (0,1){1};
\node (S2) at (0.5,0){2};
\node (S3) at (1,1) {3};
\node (S4) at (1.5,0) {4};
\node (S5) at (2,1) {5};
\node (S6) at (2.5,0) {6};
\node (S7) at (3,1) {7};
\node (S8) at (3.5,0) {8};
\node (S9) at (4,1) {9};
\node (S10) at (4.5,0) {10};
\draw (S3)--(S1)--(S2)--(S3)--(S4)-- (S2);
\draw (S3)--(S5)--(S7)--(S9)--(S10)--(S8)--(S6)--(S4);
\draw (S4)--(S5)--(S6)--(S7)--(S8)--(S9);
\end{tikzpicture}\hspace{1cm}
\begin{tikzpicture}[thick]
\node (S1) at (0,1){1};
\node (S2) at (0,0){2};
\node (S3) at (1,1) {3};
\node (S4) at (1,0) {4};
\draw (S3)--(S1)--(S2)--(S3)--(S4)-- (S2);
\end{tikzpicture}
\resizebox{4cm}{!}{
\begin{tikzpicture}[thick,scale=2]
\coordinate (A1) at (0,2);
\coordinate (A11) at (-0.39,1.5);
\coordinate (A111) at (-0.25,1.498);
\coordinate (A112) at (-0.33,1.46);
\coordinate (A12) at (0.39,1.5);
\coordinate (A121) at (0.25,1.498);
\coordinate (A122) at (0.33,1.46);
\coordinate (A13) at (0,1.25);
\coordinate (A131) at (-0.153,1.35); 
\coordinate (A132) at (0.153,1.35); 
\coordinate (A2) at (0,0); 
\coordinate (A21) at (-0.387,0.213); 
\coordinate (A211) at (-0.29,0.2795); 
\coordinate (A212) at (-0.28,0.213); 
\coordinate (A22) at (0.387,0.213); 
\coordinate (A221) at (0.29,0.2795); 
\coordinate (A222) at (0.28,0.213); 
\coordinate (A23) at (0,0.5); 
\coordinate (A231) at (-0.153,0.388); 
\coordinate (A232) at (0.153,0.388); 
\coordinate (A3) at (-1.1,0.6);
\coordinate (A31) at (-0.9,0.49);
\coordinate (A311) at (-0.88,0.59);
\coordinate (A312) at (-0.84,0.51);
\coordinate (A32) at (-0.8,0.976);
\coordinate (A321) at (-0.825,0.845);
\coordinate (A322) at (-0.745,0.875);
\coordinate (A33) at (-0.6,0.6);
\coordinate (A4) at (1.1,0.6);
\coordinate (A41) at (1.027,0.565);
\coordinate (A411) at (0.88,0.59);
\coordinate (A412) at (0.84,0.51);
\coordinate (A42) at (0.95,0.6);
\coordinate (A421) at (0.825,0.845);
\coordinate (A422) at (0.745,0.875);
\coordinate (A44) at (0.6,0.6);

\draw[draw=ForestGreen]  (A3)--(A2);
\draw[draw=ForestGreen]  (A1)--(A12);
\draw[draw=ForestGreen] (A2)--(A22);
\draw[draw=ForestGreen] (A2) -- (A1);
\draw[draw=ForestGreen] (A3)--(A1);
\draw[draw=ForestGreen,dashed]  (A33) -- (A3);
\draw[draw=ForestGreen,dashed]  (A44) -- (A4);
\draw[draw=ForestGreen] (A41)--(A4)--(A12);
 \draw[draw=black,fill=none]   (A321)--(A311);
\draw  (A212)-- (A222)--(A221);
 \draw[draw=ForestGreen] (A22) -- (A41);
 \draw (A412)--(A222);
 \draw (A411)--(A221);
\draw (A311)--(A312);
\draw (A411)--(A412);
\draw (A111)--(A121);
\draw (A111)--(A112);
\draw (A311)--(A211)--(A212) --(A312);
\draw (A112)--(A321);
\draw[dashed] (A321)--(A322)--(A111);
\draw  (A211)--(A231)--(A232)--(A221);
\draw[dashed]  (A33) -- (A44);
\draw (A231) -- (A131) -- (A132) -- (A232) -- cycle;
\draw[dashed] (A322) -- (A33) -- (A312);
\draw (A112) -- (A131) --(A132)--(A122);
\draw[dashed]  (A422) -- (A121); 
\draw (A421)--(A122) --(A121);
\draw[dashed] (A412)--(A44)--(A422)--(A421);
\draw (A421)--(A411)--(A412);
\end{tikzpicture}}
\caption{Top: Examples of friezohedra. The rightmost one is the  (compact) friezohedron on $\{1, \ldots, 10\}$ \\
Bottom: Hypergraph and truncated simplex associated with the compact friezohedron on $4$ vertices}
 \label{Figfriezoedre}
\end{figure}
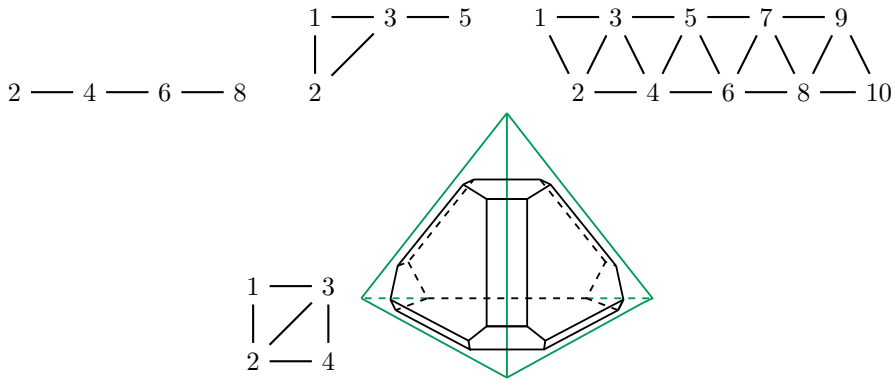

\begin{example} \label{expleFriezo}
Our final example is  the family of {\bf friezohedra}.  Consider the infinite graph $\Friezo$ on $\mathbb{Z}$ with the set of edges $\{(x,y) ||x-y|\leq 2 \}$, and its restrictions $\Friezo_X$ to finite sets $X=\{x_1<\dots<x_n\} \subseteq \mathbb{Z}$ such that $\Friezo_X$ is connected, which we call friezohedra. Note that $\Friezo_X$  is connected  exactly when there is no $i$ such that $x_{i+1}-x_i>2$,. 
We distinguish the \emph{compact friezohedra}, which are the friezohedra such that  
$X$ is an interval in $\mathbb{Z}$ (implying a fortiori that  $\Friezo_X$  is connected). Families constructed from an infinite hypergraph through restrictions as in this example are called \emph{restrictohedra} and are studied in full generality in \S\ref{restrictohedra-section}.
The name "friezohedron" comes from the shape of the hypergraphs of a compact friezohedron for $X$ sufficiently large, as illustrated in Figure \ref{Figfriezoedre}, where the associated polytope in dimension 3 is also drawn. We do not have at the time of writing a ``simple'' combinatorial interpretation of the constructs of the compact friezohedra. 
In Figure \ref{ArrFr}, we give the number of constructs with $k$ nodes for  $|X|=n$, for low values of $k,n$.

\begin{figure}
\begin{center}
\begin{tabular}{|l|c|c|c|c|c|c|}\hline
\diagbox[width=3em]{n}{k}&
1&2&3&4&5 & Sum over k\\ \hline
1 & 1 &  &  &  &  & 1 \\ \hline
2 & 1  & 1 &  & & & 2  \\ \hline
3 & 1  & 6  & 6 & & & 13   \\ \hline
4 & 1  & 13  & 33  & 22 & & 69  \\ \hline
5 & 1  & 25  & 119  & 188  & 94 & 427\\ \hline
\end{tabular} 
\end{center}

\caption{Number of constructs with $k$ vertices of the compact friezohedron on $n$ vertices \label{ArrFr}}
\end{figure}

\end{example}

More examples of truncations and constructs are to be found in \cite{DP,COI}, and also below.

\section{  shuffle product} \label{strict-teams}

In our main section, we unify the above mentioned works of Burgunder, Loday and Ronco into a notion that we call   shuffle product of constructs (defined in an unbiased style, cf. \S\ref{Intro-section}). Towards achieving this goal, in \S\ref{tcd} we introduce a general  framework based on the formalism of hypergraph polytopes of \S\ref{reminders-hypergraph-section},  which will serve as ``carrier'' of an algebraic structure that we define by induction in \S\ref{asp}. We show that the structure satisfies an equation that we call polydendriform, and derive an associative product from it. We show that associahedra and permutohedra fit in this framework, as well as all families of restrictohedra, which we define and study in \S\ref{restrictohedra-section}. We illustrate the notions introduced with the example of  friezohedra. In \S\ref{non-recursive}, we give an alternative non-recursive definition of the associated associative product. 
 In \S\ref{semi-strict-subsection}, we further enlarge  the    framework  to cover more examples.

 \subsection{Strict teams, clans and delegations}\label{tcd}
 
We first specify a collection (or {\em universe}) $\mathfrak U$ of connected hypergraphs.  Note that contrary to \cite{RoncoGTO}, some universes may contain several hypergraphs on the same set of vertices. It is for instance the case for erosohedra, see Example \ref{eroso2}.

A {\it preteam} is a  pair $\tau=(\setc{\hyper{H}_a}{a\in A},\hyper{H})$ of a finite set $\setc{\hyper{H}_a\in\mathfrak U}{a\in A}$  of hypergraphs (for some indexing set $A$) and a hypergraph $\hyper{H}\in\mathfrak U$, such that the $H_a$ are mutually disjoint and $H=\sqcup_{a\in A}H_a$.  We call $\hyper{H}$ and the $\hyper{H}_a$  the coordinating hypergraph and the participating hypergraphs, respectively. 
The idea is that, given constructs $C_a: \hyper{H}_a$ for all $a\in A$, we  aim at defining a product of 
the $C_a$ living in \hyper{H}. But we need to impose conditions on our preteams.

\begin{example} \label{explePermuto}
Let us first consider the universe formed by all permutohedra $\hyper{P}^X$. An example of preteam is given by: 
\begin{equation} \label{preteamPermuto}
\tau^\hyper{P}=\left(\{\hyper{P}^{\{\text{\ding{94}},\text{\ding{98}}\}}, \hyper{P}^{\{\text{\ding{95}}, \text{\ding{96}}\}}, \hyper{P}^{\{\text{\ding{97}}\}}, \hyper{P}^{\{\text{\ding{99}}\}}\}, \hyper{P}^{\{\text{\ding{94},\ding{95},\ding{96},\ding{97},\ding{98},\ding{99}}\}}\right).
\end{equation}
The reader may wonder why we do not simply take $X=[n]$ (for varying $n$), as in \S\ref{Prologue-section}. We refer to  Remark \ref{species-remark} for a discussion.
\end{example}

\begin{example} \label{expleFrise} 
An example of preteam for the universe of friezohedra is given by $$(\{\hyper{H}_1=\hyper{F}_{\{1,3,5\}},\hyper{H}_2=\hyper{F}_{\{2,4\}}, \hyper{H}_3=\hyper{F}_{\{6,7,8\}}\}, \hyper{F}_{\{1,\ldots,8\}}).$$
\end{example}

\begin{figure}
\begin{center}
\resizebox{0.9\textwidth}{!}{\begin{tikzpicture}
\node (a0) at (-3.8,0.5) {${\bf H}_{a_0}$};
\node (a1) at (-2.1,0.75) {${\bf H}_{a_1}$};
\node (a2) at (-0.4,0.5) {${\bf H}_{a_2}$};
\node (a3) at (0.9,0.625) {${\bf H}_{a_3}$};
\node (h) at (-1.5,-3.5) {${\bf H}$};
\draw[line width=0.1em,draw=none,fill=red!30]  (-1.5,-2.5) ellipse (1.5cm and 0.65cm);
\draw [dashed,line width=0.1em]  (-3,-2.5) arc[start angle=180,end angle=0,x radius=1.5,y radius=0.65]; 
\draw[line width=0.1em] (-3,-2.5) arc[start angle=180,end angle=360,x radius=1.5,y radius=0.65]; 
\draw[line width=0.1em,fill=yellow!30] (-3.8,0) ellipse (0.4cm and 0.25cm);
\draw[line width=0.1em,fill=cyan!30] (-2.1,0) ellipse (1cm and 0.5cm);
\draw[line width=0.1em,fill=magenta!30] (-0.4,0) ellipse (0.4cm and 0.25cm);
\draw[line width=0.1em,fill=green!30] (0.9,0) ellipse (0.6cm and 0.35cm);
\draw[line width=0.1em] (-4.2,0) to[out=down,in=up] (-3,-2.5);
\draw[line width=0.1em] (-3.4,0) to[out=down,in=down,looseness=3] (-3.1,0);
\draw[line width=0.1em] (-1.1,0) to[out=down,in=down,looseness=3] (-0.8,0);
\draw[line width=0.1em] (0,0) to[out=down,in=down,looseness=3] (0.3,0);
\draw[line width=0.1em] (1.5,0) to[out=down,in=up] (0,-2.5);
\end{tikzpicture} \quad\quad \begin{tikzpicture}
\node (a0) at (-3.8,0.5) {${\bf H}_{a_0}$};
\node (a1) at (-2.1,0.75) {${\bf H}_{a_1}\backslash X_{a_1}$};
\node (a2) at (-0.4,0.5) {${\bf H}_{a_2}$};
\node (a3) at (0.9,0.625) {${\bf H}_{a_3}\backslash X_{a_3}$};
\node (h) at (-1.5,-3.5) {${\bf H}\backslash (X_{a_1}\cup X_{a_3})$};
\draw[line width=0.1em,draw=none,fill=red!30]  (-2.2,-2.65) ellipse (0.3cm and 0.125cm);
\draw[line width=0.1em,draw=none,fill=red!30]  (-0.7,-2.65) ellipse (0.3cm and 0.125cm);
\draw[line width=0.1em,draw=none,fill=red!30]  (-1.45,-2.35) ellipse (0.3cm and 0.125cm);
\draw [dashed,line width=0.1em]  (-3,-2.5) arc[start angle=180,end angle=0,x radius=1.5,y radius=0.65]; 
\draw[line width=0.1em] (-3,-2.5) arc[start angle=180,end angle=360,x radius=1.5,y radius=0.65]; 
\draw[line width=0.1em,fill=yellow!30] (-3.8,0) ellipse (0.4cm and 0.25cm);
\draw[line width=0.1em] (-2.1,0) ellipse (1cm and 0.5cm);
\draw[line width=0.1em,fill=magenta!30] (-0.4,0) ellipse (0.4cm and 0.25cm);
\draw[line width=0.1em] (0.9,0) ellipse (0.6cm and 0.35cm);
\draw[line width=0.1em,fill=cyan!30] (-2.5,-0.1) ellipse (0.25cm and 0.125cm);
\draw[line width=0.1em,fill=cyan!30] (-1.7,-0.1) ellipse (0.25cm and 0.125cm);
\draw[line width=0.1em,fill=cyan!30] (-2.1,0.2) ellipse (0.25cm and 0.125cm);
\draw[line width=0.1em,fill=green!30] (0.9,0) ellipse (0.25cm and 0.125cm);
\draw [dashed,line width=0.1em]  (-2.5,-2.65) arc[start angle=180,end angle=0,x radius=0.3,y radius=0.125]; 
\draw[line width=0.1em] (-2.5,-2.65) arc[start angle=180,end angle=360,x radius=0.3,y radius=0.125]; 
\draw [dashed,line width=0.1em]  (-1,-2.65) arc[start angle=180,end angle=0,x radius=0.3,y radius=0.125]; 
\draw[line width=0.1em] (-1,-2.65) arc[start angle=180,end angle=360,x radius=0.3,y radius=0.125]; 
\draw [dashed,line width=0.1em]  (-1.75,-2.35) arc[start angle=180,end angle=0,x radius=0.3,y radius=0.125]; 
\draw[line width=0.1em] (-1.75,-2.35) arc[start angle=180,end angle=360,x radius=0.3,y radius=0.125]; 
\draw[line width=0.1em] (-4.2,0) to[out=down,in=up] (-3,-2.5);
\draw[line width=0.1em] (-3.4,0) to[out=down,in=down,looseness=3] (-3.1,0);
\draw[line width=0.1em] (-1.1,0) to[out=down,in=down,looseness=3] (-0.8,0);
\draw[line width=0.1em] (0,0) to[out=down,in=down,looseness=3] (0.3,0);
\draw[line width=0.1em] (1.5,0) to[out=down,in=up] (0,-2.5);
\draw[line width=0.075em] (-4.2,0) to[out=down,in=up] (-2.5,-2.65);
\draw[line width=0.075em] (-3.4,0) to[out=down,in=down,looseness=3] (-2.75,-0.1);
\draw[line width=0.075em] (-2.25,-0.1) to[out=down,in=down,looseness=3] (-1.95,-0.1);
\draw[line width=0.075em] (-1.45,-0.1) to[out=down,in=up] (-1.9,-2.65);
\draw[line width=0.075em] (0.65,0) to[out=down,in=up] (-1,-2.65);
\draw[line width=0.075em] (1.15,0) to[out=down,in=up] (-0.4,-2.65);
\draw[line width=0.075em] (0,0) to[out=down,in=up] (-1.15,-2.35);
\draw[line width=0.075em,dashed] (-2.35,0.2) to[out=down,in=up] (-1.75,-2.35);
\draw[line width=0.075em,dashed] (-1.85,0.2) to[out=down,in=down,looseness=3] (-0.8,0);
\end{tikzpicture}} \\
{\small a) \hspace{6cm} b)}
\end{center}
\caption{a) A preteam $\tau=(\{{\bf H}_{a_0},{\bf H}_{a_1}, {\bf H}_{a_2},{\bf H}_{a_3}\},{\bf H})$ is represented as a cobordism whose upper and lower boundary disks feature the participating and coordinating
hypergraphs, respectively. b) For  $X_{a_1}\subseteq H_{a_1}$ and $X_{a_3}\subseteq H_{a_3}$, the decompositions  ${\bf H}_{a_1}\backslash X_{a_1}\leadsto \hyper{H}_{(a_1,1)},\hyper{H}_{(a_1,2)},\hyper{H}_{(a_1,3)}$,  ${\bf H}_{a_3}\backslash X_{a_3}\leadsto \hyper{H}_{(a_3,1)}$ and ${\bf H}\backslash (X_{a_1}\cup X_{a_2})\leadsto \hyper{H}_1,\hyper{H}_2,\hyper{H}_3$ are represented by ``embeddings of little disks into big disks'', in such a way that the little disks represent the corresponding connected components, and their complements in big disks are the removed sets. This allows us to visualize the induced preteams $\tau,X_{a_1}\cup X_{a_2}\clandec \tau_1,\tau_2,\tau_3$ as cobordisms in the interior of $\tau$.  }\label{blah}
\end{figure}

A preteam  is called a {\it strict team} if  for each choice of a   subset $\emptyset\neq B\inc A$ and of a subset $\emptyset\neq X_b\subseteq H_b$ for each $b\in B$, inducing the decompositions
$\hyper{H}_b, X_b \leadsto \hyper{H}_{(b,1)},\ldots ,\hyper{H}_{(b,n_b)}$   and  $\hyper{H},\Union_{b\in B}X_b\leadsto \hyper{H}^B_1,\ldots,\hyper{H}^B_{n_B}$,
we have that, for each $$\tilde{a}\in \tilde{A}:=(A\backslash B)\union\setc{(b,i)}{b\in B, 1\leq i\leq n_b},$$ 
$\hyper{H}_{\tilde{a}}\in\mathfrak U$, and  $H_{\tilde{a}}$ is included  in a connected component  of $\restrH{H}{(\Union_{b\in B}X_b)}$. As we shall see, preteams associated respectively with associahedra, permutohedra and friezohedra are strict. 
On the other hand,  preteams associated with simplices, erosohedra and hypercubes are not strict, as some $H_{\tilde{a}}$ are not included in a connected component of    $\restrH{H}{(\Union_{b\in B}X_b)}$:
these last examples fit in the formalism of semi-strict teams introduced in  \S\ref{semi-strict-subsection}.

\begin{lemma} \label{Jovana-lemma} 
A preteam $(\setc{\hyper{H}_a}{a\in A},\hyper{H})$ is strict iff, for all $a \in A$ and $e\in\hyper{H}_a$,  $e$ is connected in $\hyper{H}$. Also, in the above definition of strict team, it holds that $H_{\tilde{a}}$ is connected in $\hyper{H}$, for all $\tilde{a}\in\tilde{A}$.
\end{lemma}
\begin{proof}  We shall prove the equivalence $(1)\Leftrightarrow (2)\Leftrightarrow (3)$, where (1) is the definition of strict team given above, (2) is the characterization claimed in the statement, and (3) is the definition of team above enhanced with the additional property claimed in the statement.
\begin{itemize}
\item $(1)\Rightarrow (2)$. If $(\setc{\hyper{H}_a}{a\in A},\hyper{H})$ is a strict team in the sense of the definition given above, then, in particular, for each $a \in A$ and $e\in\hyper{H}_a$, taking $B=A$, $X_b=H_b$ for $b\neq a$ and $X_a=(H_a\backslash e)$, we get that $\hyper{H}_a,X_a\leadsto \hyper{H}_e$, and hence that $e$ is included in a connected component $K$ of $\restrH{H}{(\Union_{b\in B}X_b)}$. But for our choice of $B$, we have
$H\backslash(\Union_{b\in B}X_b)=e$, hence this forces $K=e$, and a fortiori $e$ is connected in $\hyper{H}$. 

\item $(2)\Rightarrow (3)$. Let $\tilde{a}\in\tilde{A}$ and $\tilde{e}\in \hyper{H}_{\tilde{a}}$. Then a fortiori $\tilde{e}\in \hyper{H}_{\pi(a)}$, where $\pi:\tilde{A}\rightarrow A$ is defined by $\pi(\tilde{a})=\tilde{a}$ if $\tilde{a}\in  A\backslash B$ and $\pi{(b,i)}=b$. Since we assume (2), we have that $\tilde{e}$ is connected in $\hyper{H}$.  Thus, all hyperedges of $\hyper{H}_{\tilde{a}}$ are connected in $\hyper{H}$. By standard connectedness arguments, this, together with the fact that $\hyper{H}_{\tilde{a}}$ is connected, 
implies that $H_{\tilde{a}}$ is connected in $\hyper{H}$: informally, every path of hyperedges of $\hyper{H}_{\tilde{a}}$ witnessing the connectedness of $\hyper{H}_{\tilde{a}}$, for arbitrary chosen vertices in  $H_{\tilde{a}}$, can be turned into a path of  hyperedges of $\hyper{H}$ witnessing the connectedness of $H_{\tilde{a}}$ in $\hyper{H}$ for the same chosen vertices.

\item $(3)\Rightarrow (1)$. Obvious.
\end{itemize}
\end{proof}

Note that, for each $\emptyset\neq B\subseteq A$ and a choice of $\emptyset\neq X_b\subseteq H_b$ for each $b\in B$,   the structure of a strict team $\tau$ implies the existence of a surjective function 
$$\varphi^{B,\setc{X_b}{b\in B}}_{\tau}:\tilde{A}\rightarrow \{1,\dots,n_B\} \quad(\mbox{written}\; \varphi^{B}_{\tau}\;\mbox{for short}),$$  which associates to  $\tilde{a}\in \tilde{A}$ the  index   of the connected component  of ${\bf H}\backslash \bigcup_{b\in B}X_b$ that contains $H_{\tilde{a}}$.  By Lemma \ref{Jovana-lemma}, this determines  preteams $$\tau_i =(\setc{\hyper{H}_{\tilde{a}}}{\tilde{a}\in \tilde{A} \mbox{ and } \varphi_{\tau}(\tilde{a})=i},\hyper{H}^B_i) \quad (1\leq i\leq n_B).$$ 
We summarize this by the notation $\tau,\Union_{b\in B}X_b\clandec \tau_1,\ldots,\tau_{n_B}$. 
 
\begin{example} 
\label{expleFrise2} Consider the preteam in Example \ref{expleFrise} :
$$(\{\hyper{H}_1=\hyper{F}_{\{1,3,5\}},\hyper{H}_2=\hyper{F}_{\{2,4\}}, \hyper{H}_3=\hyper{F}_{\{6,7,8\}}\}, \hyper{F}_{\{1,\ldots,8\}})$$
and consider $B=\{1,2\}$, $X_1=\{3\}$ and $X_2=\{2\}$, inducing the decompositions $\hyper{H}_1, X_1 \leadsto \hyper{H}_{(1,1)}=\hyper{F}_{\{1\}}, \hyper{H}_{(1,2)}=\hyper{F}_{\{5\}}$, $\hyper{H}_2, X_2 \leadsto \hyper{H}_{(2,1)}=\hyper{F}_{\{4\}}$ and $\hyper{H}, \cup_{i \in B} X_i \leadsto \hyper{H}_1^B=\hyper{F}_{\{1\}}, \hyper{H}_2^B=\hyper{F}_{\{4, \ldots, 8\}}$. The map $\varphi^{B,\setc{X_b}{b\in B}}_{\tau}$ associates $1$ to $(1,1)$ and $2$ to the other elements. This leads to two preteams $\tau_1=(\{\hyper{F}_{\{1\}}\},\hyper{F}_{\{1\}})$ and $\tau_2=(\{ \hyper{F}_{\{4\}}, \hyper{F}_{\{5\}}, \hyper{F}_{\{6,7,8\}}\}, \hyper{F}_{\{4, \ldots, 8\}})$. 
\end{example}

 A {\it  strict clan} is a set $\Xi$ of strict teams such that, for each team $\tau\in \Xi$, and each situation $\tau,\Union_{b\in B}X_b\clandec \tau_1,\ldots,\tau_n$ as above, we have that $\tau_i\in \Xi$ for all $i$.   In order to ease the understanding of the decomposition $\tau,\Union_{b\in B}X_b\clandec \tau_1,\ldots,\tau_{n_B}$,   in Figure \ref{blah}, we suggest an interpretation of preteams and strict teams in terms of cobordisms. 
 \smallskip

Let us fix a strict clan $\Xi$, and   some $q\in{\mathbb R}$ (our product will be parameterized by $q$, cf. end of \S\ref{Prologue-section}).
A  $\Xi$-{\it delegation} (or delegation for short) is  a pair 
$$\delta=(\{C_a:{\bf H}_a\,|\, a\in A\},{\bf H}) \quad\mbox{
such that}\quad \tau:=(\setc{\hyper{H}_a}{a\in A},\hyper{H})\in \Xi.$$
We say that $\tau$ is the support of $\delta$, and that $C_a$ is the construct of $\delta$ at position $a$. 
 Observe that, for  $\emptyset \neq B\subseteq A$ and $\tilde{A}$  as above, assuming that $X_a$ is the root vertex of $C_a$ for each $a\in A$, there is a canonical association of a construct $C_{\tilde{a}}$ to each $\tilde{a}\in \tilde{A}$, which gives rise to delegations \begin{equation}\label{deltas}\delta^B_i=(\{C_{\tilde{a}}:{\bf H}_{\tilde{a}}\,|\, \tilde{a}\in \tilde{A} \mbox{ and }\varphi^B_\tau(\tilde{a})=i\},{\bf H}_i^B),\end{equation} for $1\leq i\leq n_B$.  More precisely, for $b\in B$, we set $C_b=X_b(C_{(b,1)},\ldots,C_{(b,n_b)})$ with $C_{(b,i)}:\hyper{H}_{(b,i)}$. We summarize this by the notation $\delta,\Union_{b\in B} X_b\deldec \delta^B_1,\ldots,\delta^B_{n_B}$. 
 
\begin{example} \label{delegPermuto}
The strict clan associated with permutohedra is obtained by considering the set of all  preteams $$(\{\hyper{P}^{V_i}\}_{i \in I}, \hyper{P}^V),$$ where $\{V_i\}_{i \in I}$ forms a partition of $V$ (in the universe of permutohedra, it is easily checked that all preteams are in fact strict).
\end{example}

\begin{example} \label{expleFriseClan}
The strict clan associated with friezohedra is obtained by considering the set of all strict teams $$(\{\hyper{F}_{V_i}\}_{i \in I}, \hyper{F}_V),$$ where $\{V_i\}_{i \in I}$ forms a partition of $V$ and each hypergraph $\hyper{F}_{V_i}$ and $\hyper{F}_{V}$  are connected. A delegation associated with the strict team of Example \ref{expleFrise} is given by:
$$(\{3(1,5) : \hyper{F}_{\{1,3,5\}},2(4) : \hyper{F}_{\{2,4\}} ,678 : \hyper{F}_{\{6,7,8\}}\}, \hyper{F}_{\{1,\ldots,8\}}) .$$

 Considering $B=\{1,2\}$, $X_1= \{3\}$ and $X_2= \{2\}$ as in Example \ref{expleFrise2}, we get delegations 
\begin{align*}
\delta_1^B&=(\{1 : \hyper{F}_{\{1\}}\}, \hyper{F}_{\{1\}}) \\
\text{and } \delta_2^B&=(\{ 4 : \hyper{F}_{\{4\}}, 5 : \hyper{F}_{\{5\}}, 678 : \hyper{F}_{\{6,7,8\}}\}, \hyper{F}_{\{4,\ldots, 8\}}) .
\end{align*} 

\end{example}

 We end this section by defining  further conditions on clans:
 \begin{itemize}
 \item  A clan  $\Xi$ is  {\em associative} if,
for all $$\quad\quad\quad\tau \!=\! (\setc{\hyper{H}_a}{a\in A},\hyper{H})\in \Xi\:,\:a_0\in A\: ,\:
\tau' \!=\! (\setc{\hyper{H}_{(a_0,a')}}{a'\in A'},\hyper{H}_{a_0}) \in\Xi,$$ we have
$$\tau'':=(\{{\bf H}_{a}\,|\,a\in A\backslash \{a_0\}\}\cup \{\hyper{H}_{(a_0,a')}\,|\,a'\in A'\},{\bf H})\in\Xi.$$ We shall refer to  $\tau''$ as the grafting of $\tau'$ to  $\tau$ along $a_0$.   (Note that, again, we set the scene here for an unbiased version of associativity). We shall need this condition in order to phrase and prove the associativity of the product that we define in \S\ref{asp}.
\item In a different direction, we define the notion of {\em ordered} (strict) universe, preteam, team and clan. We suppose given an ordered set, say $\mathbb{Z}$. For $X_1,X_2\inc\mathbb{Z}$, we write $X_1<X_2$ if $\max(X_1)<\min(X_2)$.  An \emph{ordered universe} is a universe $\mathfrak{U}$ such that, for all $\hyper{H}\in\mathfrak{U}$, $H\inc\mathbb{Z}$,  and such that  all decompositions 
$\hyper{H},X\leadsto \hyper{H}_1,\ldots,\hyper{H}_p$ can be indexed in such a way that $H_i<H_{i+1}$ for all $i$.  An \emph{ordered preteam} is a pair $((\hyper{H}_1,\ldots,\hyper{H}_p),\hyper{H})$ such that $(\set{\hyper{H}_1,\ldots,\hyper{H}_p},\hyper{H})$ is a preteam and such that $H_1<\cdots<H_p$. Ordered teams are teams whose underlying preteam is ordered. Note that when $\tau$ is ordered, if $\tau,\Union_{b\in B}X_b\clandec \tau_1,\ldots,\tau_{n_B}$, then each $\tau_i$  is ordered (to see this, one uses the assumption that $\mathfrak{U}$ is ordered).
An ordered clan is a clan whose teams are all ordered. 
\end{itemize}

\subsection{Restrictohedra} \label{restrictohedra-section}

 Our main provision of strict clans comes from the universes of restrictohedra, that we define next.
Fix a (possibly infinite) hypergraph $\hyper{R}$, and let $\mathfrak{U}_{\hyper{R}}$ be the universe consisting of all hypergraphs $\hyper{R}_X$,   such that $X\inc R$ is non-empty and finite, and  $\hyper{R}_X$ is connected: we call them the $\hyper{R}$-{\em restrictohedra}, or restrictohedra for short. Let $\Xi_\hyper{R}$ be the set of all pairs $\left( \{\hyper{R}_{V_a}|a\in A\}, \hyper{R}_V\right)$
where $V\inc R$,
 $\{V_a\}_{a \in A}$ forms a partition of $V$, and the hypergraphs $\hyper{R}_{V_a}$ and $\hyper{R}_V$ are all in $\mathfrak{U}_\hyper{R}$.
 We can restrict this to an ordered setting if $\hyper{R}$ is {\em order-friendly}, meaning that $R\subseteq\mathbb{Z}$ and that the connected components $\hyper{R}_{V_1},\ldots,\hyper{R}_{V_p}$ of $\hyper{R}_V$, for any finite $V\inc\mathbb{Z}$ such that $\hyper{R}_V$ is not connected, can be indexed in such a way that $V_i<V_{i+1}$ for all $i$. 
 
 \begin{proposition} \label{Xi-K-strict-associative}
For all $\hyper{R}$, $\Xi_\hyper{R}$ is an associative clan.  If $\hyper{R}$ is order-friendly, then the restriction of $\Xi_\hyper{R}$ (still denoted by $\Xi_\hyper{R}$) to its ordered preteams is an ordered associative clan.
\end{proposition}
\begin{proof}  
We first note that 
every preteam $(\{\mathbf{R}_{V_a} | a \in A\}, \mathbf{R}_V)$ satisfies $\Union_{a\in A} \mathbf{R}_{V_a} \subseteq \mathbf{R}_V$ by definition, and hence, by  Lemma \ref{Jovana-lemma}, is a fortiori a strict team. Next, if (in the notation of \S\ref{tcd})
$\tau\backslash\Union_{b\in B}X_b\clandec \tau_1,\ldots,\tau_{n_B},$ we have to prove that $\tau_i\in \Xi_\hyper{R}$ for all $i$. This follows from the fact that, for any $V$ and $W\inc V$, $(\hyper{R}_V)\backslash W=\hyper{R}_{V\backslash W}$ and that, for all $X\inc R$, the connected components of $\hyper{R}_X$ are all of the form $\hyper{R}_Y$ for some $Y\inc X$.  Finally, the clan is associative since $\Xi_\hyper{R}$ includes all ``possible'' preteams in the sense that  for any $X\inc R$ and any partition $\setc{X_a}{a\in A}$ of $X$, we have $(\{\mathbf{R}_{X_a} | a \in A\}, \mathbf{R}_X)\in  \Xi_\hyper{R}$ if and only if  $\mathbf{R}_X$ and $\mathbf{R}_{X_a}$ (for all $a\in A$) are connected.

Suppose now that $\hyper{R}$ is moreover order-friendly. Then  it is immediate that $\mathfrak{U}_{\hyper{R}}$ is ordered. Since we limit ourselves to ordered preteams  $((\hyper{H}_1,\ldots,\hyper{H}_m),\hyper{R}_V)$ with  $\hyper{H}_i=\hyper{R}_{V_i}$ and $V_i<V_{i+1}$ for all $i$, and since $\hyper{R}$ is order-friendly, then for all $B\inc A=\set{1,\ldots,m}$ there is an induced order on $\tilde{A}$ such that, if $\tilde{a_1}<\tilde{a_2}$, then $V_{\tilde{a_1}}<V_{\tilde{a_2}}$, where $\hyper{H}_{\tilde{a}}=\hyper{R}_{V_{\tilde{a}}}$. This in turn implies that $(\varphi^B_\tau)^{-1}(1)$,\ldots,  $(\varphi^B_\tau)^{-1}(n_B)$ form successive intervals of $\tilde{A}$, and hence that each $\tau_i$ is ordered.
\end{proof}

The following family of graphs provides examples of order-friendly graphs (and hence of ordered associative clans).
\begin{proposition} \label{Gamma-order-friendly}
For all   $1\leq k\in \mathbb{N}\cup\set{\infty}$, the following graph is order-friendly:
$$\bold\Gamma^k:=\setc{\set{a}}{a\in\mathbb{Z}}\union\setc{\set{a,a+l}}{a\in\mathbb{Z},l\in\mathbb{N},1\leq l\leq k}.$$
\end{proposition}
\begin{proof}
A subset $V\inc\mathbb{Z}$ is not connected in  $\bold\Gamma^k$ if  and only if there is a set $X$ of at least $k$ consecutive integers in $]\min(V);\max(V)[$, which does not intersect $V$.  
If $X_1,\ldots,X_p$ are the sets of such maximal sequences of consecutive integers, then the interval $[\min(V),\max(V)]$ in $\mathbb{Z}$ is the union of consecutive intervals $I_0,X_1,I_1,\ldots,X_p,I_p$, and the connected components of $(\bold\Gamma^k)_V$ are 
$(\bold\Gamma^k)_{V\cap I_0},\ldots,(\bold\Gamma^k)_{V\cap I_p}$. Then $(V\cap I_j)<(V\cap I_{j+1})$ follows a fortiori from $I_j<I_{j+1}$.
\end{proof}
By Propositions \ref{Gamma-order-friendly} and \ref{Xi-K-strict-associative}, we get an induced associative ordered clan $\Xi_{\bold\Gamma^k}$. 

In the extreme cases $k=1$ and $k=\infty$,  we have our old friends $\bold\Gamma^1_X=\hyper{K}^X$ (for $X$ interval of $\mathbb{Z}$) and $\bold\Gamma^\infty_X=\hyper{P}^X$ (for finite $X\inc\mathbb{Z}$), respectively.
The teams are of the form $(\{\bold\Gamma^1_{X_1},\ldots,\bold\Gamma^1_{X_p}\},\bold\Gamma^1_{\Union X_i})$ (where the $X_i$ are adjacent  intervals)  and $(\{\bold\Gamma^\infty_{X_1},\ldots,\bold\Gamma^\infty_{X_p}\},\bold\Gamma^\infty_{\Union X_i})$ (where $X_i<X_{i+1}$ for all $i<p$), respectively.  For $k=2$,
 we have $\bold\Gamma^2=\hyper{F}$, and hence we recover also friezohedra as a special case.

\smallskip
We end the section with a  characterization of universes arising as restrictohedra.
 
\begin{proposition}  \label{restrictohedra-characterisation}
A universe $\mathfrak{U}$ is of the form $\mathfrak{U}_{\hyper{R}}$, for some hypergraph $\hyper{R}$, if and only if it satisfies the  following four conditions:

\begin{enumerate}
\item
For any hypergraphs $\hyper{H}_1$ and $\hyper{H}_2$ in $\mathfrak{U}$, if $H_1=H_2$, then $\hyper{H}_1=\hyper{H}_2$.
\item  
If $\hyper{H}\in\mathfrak{U}$ and $e\in\hyper{H}$, if  $\hyper{G}\in\mathfrak{U}$ is such that $e\inc G$, then $e\in\hyper{G}$.

\item
If  $\hyper{H}\in\mathfrak{U}$, and if  $X \inc H$ is such that $\hyper{H}_X$ is connected, then there exists $\hyper{G}\in \mathfrak{U}$ such that $G=X$.

\item
If $\hyper{H}_1,\hyper{H}_2\in \mathfrak{U}$ are such that $H_1 \cap H_2$ is non-empty, then there exists $\hyper{H}\in \mathfrak{U}$  such that  $\hyper{H}_1,\hyper{H}_2\inc\hyper{H}$. 
\end{enumerate}
\end{proposition}
\begin{proof} We first check that any universe of the form $\mathfrak{U}_{\hyper{R}}$ satisfies the conditions in the statement. Condition (1) is immediate. Conditions (2), (3) and (4)  follow immediately from the observations that, by definition, for arbitrary $X$, we have  $e\in\hyper{R}_X$ if and only if $e\in\hyper{R}$ and $e\inc X$, that 
$(\hyper{R}_H)_X=\hyper{R}_X$, and that the union of two connected sets with a non-empty intersection is connected.

\smallskip
Conversely, suppose that $\mathfrak{U}$ satisfies the four conditions of the statement. We set $\hyper{R}=\Union\setc{\hyper{H}}{\hyper{H}\in\mathfrak{U}}$.  We shall show the following two properties, which (together with (1)) imply immediately that $\mathfrak{U}
=\mathfrak{U}_{\hyper{R}}$.
\begin{enumerate}
\item[(a)] If $X$ is a finite set such that there exists a hypergraph $\hyper{H}$ such that $H=X$ and $\hyper{H}\in\mathfrak{R}$, then there exists a hypergraph 
$\hyper{H}'\in \mathfrak{U}_{\hyper{K}}$ such that $H'=X$ and $\hyper{H}\inc\hyper{H}'$.
\item[(b)] If $X$ is a finite set such that there exists a hypergraph $\hyper{H}'$ such that $H'=X$ and $\hyper{H}'\in\mathfrak{U}_{\hyper{R}}$, then there exists a hypergraph 
$\hyper{H}\in \mathfrak{U}$ such that $H=X$ and  $\hyper{H}'\inc\hyper{H}$.
\end{enumerate}
For (a), we note that $\hyper{H}\inc \hyper{R}$ by definition of $\hyper{R}$, hence $\hyper{H}\inc \hyper{R}_H$, so we can set  
$\hyper{H}':= \hyper{R}_H$, noticing that $\hyper{R}_H$ is connected since  it contains a connected hypergraph (namely $\hyper{H}$) with the same set of vertices. 

We now proceed to prove (b).  By definition of $\mathfrak{U}_{\hyper{R}}$, the assumptions of (b) can be rephrased as saying that $\hyper{H}'=\hyper{R}_X$ is connected. Also,
by definition of $\hyper{R}$, for each $e\in \hyper{R}_X$, there exists a hypergraph $\hyper{H}^e\in\mathfrak{U}$ such that $e\in \hyper{H}^e$.
So we have $\hyper{R}_X\inc \Union \setc{\hyper{H}^e}{e\in \hyper{R}_X}$, this union being finite since $X$ is.  Suppose that $\hyper{R}_X$ has more than one hyperedge and pick  $e_0\in\hyper{R}_X$. We claim that there exists $e_1\in\hyper{R}_X$ such that 
$\hyper{H}^{e_0}\cap \hyper{H}^{e_1}$ is non-empty. If it were not the case, then  $\hyper{R}_X$ would be the disjoint union of $\hyper{R}_X\cap \hyper{H}^{e_0}$ and of $\hyper{R}_X \bigcap (\Union\setc{\hyper{H}^e}{e\neq e_0})$, which would  contradict the connectedness of $\hyper{R}_X$.  
We can thus replace $\setc{\hyper{H}^e}{e\in \hyper{R}_X}$ by $\setc{\hyper{H}^e}{e\neq e_0,e_1}\union\set{\hyper{H}_{01}}$, where $\hyper{H}_{01}\in\mathfrak{U}$ is obtained from $\hyper{H}^{e_0}$ and  $\hyper{H}^{e_1}$ by applying (4). By iterating this, we obtain a hypergraph $\hyper{H}'\in\mathfrak{U}$ such that $\hyper{R}_X\inc \hyper{H'}$. Note that we can write this as well as
$\hyper{R}_X\inc \hyper{H'}_X$, and, as above, we have that the connectedness of $\hyper{R}_X$ implies the connectedness of 
$\hyper{H'}_X$.

Our next (independent) observation is that in the presence of (2), condition (3) can be reinforced as follows. If  $\hyper{H}\in\mathfrak{U}$ and if $X\inc H$ is such that $\hyper{H}_X$ is connected, then there exists $\hyper{G}\in\mathfrak{U}$ such that $G=X$ and $\hyper{H}_X\inc\hyper{G}$.  Indeed, let $\hyper{G}$ be obtained by applying (3), and let $e\in\hyper{H}_X$: then this latter assumption reads as $e\inc G$, and hence $e\in \hyper{G}$ by (2).

Coming back to the proof of (b), we can apply the reinforced version of (3) to deduce the existence of a hypergraph $\hyper{H}\in\mathfrak{U}$ such that
$H=X$ and $\hyper{H'}_X\inc\hyper{H}$. We thus have $\hyper{R}_X\inc\hyper{H'}_X\inc \hyper{H}$, which concludes the proof.

\end{proof}

 \subsection{  Shuffle product of delegations of strict clans}\label{asp}
 
We now define the   shuffle product $\ast(\delta)$, for a $\Xi$-delegation $\delta$, where  $\Xi$ is a strict clan. Until \S\ref{semi-strict-subsection}, we shall omit  the adjective ``strict'' for brevity, but its presence is understood. 

A {\it linear construct} of a hypergraph ${\bf H}$ is an element of the vector space spanned by all the constructs of ${\bf H}$. We shall denote linear constructs with bold capital letters, e.g., ${\bf C}=\Sigma_{i\in I} {\lambda_i}C_i$, where $C_i:{\bf H}$, for each $i\in I$, and the notation ${\bf C}:{\bf H}$ will mean that ${\bf C}$ is a linear construct of ${\bf H}$.   We then define $X({\bf C}_1,\dots,\sum_{i\in I}\lambda_i C^i_{j},\dots,{\bf C}_n)$ as $\sum_{i\in I}\lambda_i X({\bf C}_1,\dots,C^i_{j},\dots,{\bf C}_n)$.  A {\em rooted linear construct} is a linear construct of the form  ${\bf C}=X\setc{{\bf C}_a}{a\in A}$, and we write $\mbox{root}({\bf C})=X$.

\smallskip
 The  {\it   shuffle product}  (or   product) of a  delegation $\delta=(\{C_a:{\bf H}_a\,|\, a\in A\},{\bf H})$, with $\mbox{root}(C_a)=X_a$ for all $a\in A$,  is  the  linear construct of  $\hyper{H}$ defined recursively as follows (with $\delta^B_1,\ldots, \delta^B_{n_{B}}$  as in \eqref{deltas}):
\begin{equation}\label{shuffle-def} \ast(\delta)= \sum_{\emptyset\incs B\subseteq A} q^{{|B|}-1}\ast_B(\delta), \quad \mbox{where}\;
\ast_B(\delta)=(\Union_{b\in B}X_b)(\ast(\delta^B_1),\dots,\ast(\delta^B_{n_{B}})).
\end{equation}
The instantiations of this shuffle product  to associahedra and permutohedra are  the ones recalled  in \S\ref{Prologue-section}. We detail the case of permutohedra in the next example.

\begin{example} \label{permutohedra- }
We restrict ourselves to teams with only two participating hypergraphs (which corresponds to the usual binary product on permutohedra). Then 
the   shuffle product of a delegation $$\delta=(\{C_1 : \hyper{P}^Y, C_2 :\hyper{P}^Z\},\hyper{P}^X),\;\mbox{where}\; X=Y \cup Z,
C_1=X_1(C'_1),\; \mbox{and}\; C_2=X_2(C'_2),$$
rewrites as:
\begin{align*}\label{shuffle-perm} \ast(\delta)&= \ast_{\{1\}}(\delta) + \ast_{\{2\}}(\delta)+q\ast_{\{1,2\}}(\delta), \; \mbox{where}\\
\ast_{\{1\}}(\delta)&=X_1(\ast(\{C'_1 : \hyper{P}^{Y \backslash X_1}, C_2 : \hyper{P}^Z\}, \hyper{P}^{X \backslash X_1})), \\
\ast_{\{2\}}(\delta)&=X_2(\ast(\{C_1 : \hyper{P}^{Y }, C'_2 : \hyper{P}^{Z \backslash X_2}\}, \hyper{P}^{X \backslash X_2})),\\
\ast_{\{1,2\}}(\delta)&=\left(X_1 \cup X_2 \right)(\ast(\{C'_1 : \hyper{P}^{Y \backslash X_1}, C'_2 : \hyper{P}^{Z \backslash X_2}\}, \hyper{P}^{X \backslash \left(X_1 \cup X_2 \right)})).
\end{align*}
 Writing $\ast_{\{1\}}\!=\:\prec$, $\ast_{\{2\}}\!=\:\succ$ and $\ast_{\{1,2\}}\!=\!\bcdot$, and using an infix notation, the formula for the   shuffle product on permutohedra  on two constructs $C_1=X_1(C'_1)$ and $C_2=X_2(C'_2)$ writes as:
\begin{align*} C_1 \ast C_2&= C_1 \prec C_2+C_1 \succ C_2+q (C_1 \bcdot C_2), \; \mbox{where} \\
C_1 \prec C_2 &=X_1(C'_1 \ast  C_2 ), \\
C_1 \succ C_2 &=X_2(C_1 \ast C'_2),\\
C_1 \bcdot C_2&=\left(X_1 \cup X_2 \right)(C'_1 \ast C'_2),
\end{align*}
with the convention that if $C'_1$ or $C'_2$ is the empty construct, then its   shuffle product with another  construct $C$ is $C$. It can be checked by direct induction that this  definition coincides with the one given in  \S\ref{Prologue-section}.

\begin{remark} \label{species-remark}
 Let us now explain in a few words why we found convenient to consider permutohedra on a given $X$ which is not necessarily  the usually considered set $\set{1,\ldots,n}$).  Consider the basic example:
\begin{equation}
(1,2)\bcdot (2,1) = (1,2,2,1)+ (1,3,3,2) + (2,3,3,1).
\end{equation}
In terms of constructs, this rewrites as:
\begin{equation} \label{explePermuto}
\{2\}(\{1\})\bcdot \{3\}(\{4\}) = \{2,3\} (\{1\}) \ast \{4\}).
\end{equation}
 Equation \ref{explePermuto} invites us to compute products of constructs in the complete graph on $\{1,4\}$, rather than doing some renamings.
 This  is naturally in phase with the general philosophy of species. The assignment that maps $X$ to the set of constructs of $\hyper{P}^X$
 is functorial (with respect to finite sets and bijections), giving rise to a species in the sense of Joyal. 
 \end{remark}

\end{example}

\begin{example} \label{expleFriseProduit}  As a second example, we consider  friezohedra.  Consider the delegation $$\delta=(\{2(1)\!:\! \hyper{F}_{\{1,2\}}\,,\, 3(4)\!:\! \hyper{F}_{\{3,4\}},  \}\,,\, \hyper{F}_{\{1,\ldots,4\}}).$$ The associated   shuffle product is given by:
\begin{align*}
\ast(\delta)= &\;2(1(3(4))) + 2(3(1,4)) + 2(13(4))+ 3(4(2(1)))\\
& +3(2(1,4)) + 3(24(1)) + q 23(1,4).
\end{align*}
Consider now the delegation 
\begin{equation*}
(\{3(1,5) \!:\! \hyper{F}_{\{1,3,5\}}\,,\,2(4) \!:\! \hyper{F}_{\{2,4\}} \,,\,678 \!:\! \hyper{F}_{\{6,7,8\}}\}\,,\, \hyper{F}_{\{1,\ldots,8\}}).
\end{equation*}
of Example \ref{expleFriseClan}.
 The associated   shuffle product is too big to be written here. Let us focus on the term associated with $B=\{1,2\}$. 
 We have  $\ast_B(\delta)=23(1, \ast(\delta_2^B))$, with  $\delta_2^B=(\{ 4 \!:\! \hyper{F}_{\{4\}}\,,\, 5 \!:\! \hyper{F}_{\{5\}}\,,\, 678 \!:\! \hyper{F}_{\{6,7,8\}}\}\,,\, \hyper{F}_{\{4,\ldots, 8\}})$  (as already seen in that example). By definition, we can express $\ast(\delta_2^B)$   as a sum over $\emptyset \incs B' \inc \{1,2,3\}$. Let us  again make a focus, say on $B'=\set{3}$. We get 
$$\ast_{B'}(\delta_2^B)=678(\ast(\{4 \!:\! \hyper{F}_{\{4\}}\,,\, 5 \!:\! \hyper{F}_{\{5\}}\}\,,\, \hyper{F}_{\{4,5\}})) 
 = 678(4(5)) + 678(5(4)) + q \,678(45).$$

\end{example}

When dealing with the associativity of the   product in Theorem \ref{assoc_strict} below, we shall have to take    products of (delegations made of)  linear constructs, which is not a problem, as the above definitions of $\ast$, $\ast_B$ of course extend by linearity (with the notion of delegation accordingly extended to linear constructs). 
 The following lemmas show two situations in which the linear extension of $\ast_B$ still satisfies its ``defining'' equation \ref{shuffle-def}
(now a property!). To see the need for such  lemmas, note that the definitions of the delegations $\delta^B_i$ do depend on the root of the constructs $C_b$ ($b\in B$), which no longer exists if $C_b$ is replaced by a linear construct that is not rooted.

 \begin{lemma}\label{rooted-linear-construct-in-product}
Let $(\setc{\hyper{H}_a}{a\in A},\hyper{H})$ be a strict team, and suppose that we are given rooted linear constructs ${\bf C}_a$ for each $a\in A$ with root $X_a$, forming a delegation $\delta$ (in the extended sense). Let $\emptyset \subset B\inc A$ and  let $X_B=\Union_{b\in B}X_b$.  Then we have $\ast_B(\delta) =(\Union_{b\in B}X_b)(\ast(\delta^B_1),\dots,\ast(\delta^B_{n_{B}}))$, with the same 
definition of  $\delta^B_i$ as above. 
\end{lemma}
\begin{proof} We first notice that we can indeed still define $\delta_i^B$ as before, since the only information used on constructs are their roots.
Let us  assume for simplicity that only one of the ${\bf C}_a$, say ${\bf C}_{b_0}$,  is a  rooted linear construct, all the others being plain constructs,  and that $b_0\in B$,  as  Lemma \ref{linear-construct-in-product} will a fortiori cover the case where $b_0\nin B$. We shall also assume for simplicity that 
 ${\bf C}_{b_0}=X_{b_0}\setc{{\bf C}_{b_0,i}}{1\leq i \leq n_{b_0}}$, where only one of the ${\bf C}_{b_0,i}$, say 
${\bf C}_{b_0,i_0}= \Sigma_{k\in K}  \lambda_k C_{b_0,i_0,k} $, is a linear construct, all the others being plain constructs (and we write then
${\bf C}_a=C_a$ for $a\neq b_0$ and ${\bf C}_{b_0,i}= C_{b_0,i}$ for $i\neq i_0$). Then, by ``outward'' linearity, we can write 
${\bf C}_{b_0}=\Sigma_{k\in K} \lambda_k C_{b_0,k}$, where 
$C_{b_0,k}= X_{b_0}(\setc{{C}_{b_0,i}}{i\neq i_0} \Union \set{C_{b_0,i_0,k}})$. We have
$$\ast_B(\setc{C_a}{a \in A\backslash\set{b_0}}\Union 
\set{{\bf C}_{b_0}},\hyper{H})= \Sigma_{k\in I} \lambda_k \ast_B(\setc{C_a}{a \in A\backslash\set{b_0}}\Union 
\set{C_{b_0,k}},\hyper{H}).$$
By definition,
we have
$$\ast_B(\setc{C_a}{a \in A\backslash\set{b_0}}\Union 
\set{{\bf C}_{b_0}},\hyper{H})= (\Sigma_{k\in K}  \lambda_k X_k(\ast((\delta^k)^B_1),\dots,\ast((\delta^k)^B_{n_{B}}))),$$
where for all $j\neq j_0=\varphi_\tau(b_0,i_0)$, all $(\delta^k)^B_j$ are equal to $\delta^B_j$, and where the $(\delta^k)^B_{j_0}$ differ only in one (and the same) position (the one indexed by $(b_0,i_0)$), filled  with $C_{b_0,i_0,k}$. Then we conclude by  applying ``inward'' linearity.

\end{proof}

\begin{lemma} \label{linear-construct-in-product}
Let $(\setc{\hyper{H}_a}{a\in A},\hyper{H})$ be a strict team, and let $a_0\in A$, and suppose that we are given constructs $C_a:\hyper{H}_a$ with root $X_a$ for all $a\neq a_0$, and a linear construct ${\bf C}_{a_0}$. Let $B\inc A\backslash\set{a_0}$, and let $X_B=\Union_{b\in B}X_b$.  
Then we have $\ast_B(\delta) =(\Union_{b\in B}X_b)(\ast(\delta^B_1),\dots,\ast(\delta^B_{n_{B}}))$, with the same 
definition of the teams $\delta^B_i$ as above. 
\end{lemma}

\begin{proof} The proof goes  like in Lemma \ref{rooted-linear-construct-in-product}. The only difference is that, under the assumption that $a_0\nin B$, no information at all is required on ${\bf C}_{a_0}=\Sigma_{k\in K}  \lambda_k C_{a_0,k}$. 
\end{proof}

\noindent
So far, we have a magmatic unbiased notion of product.
The following theorem establishes the associativity of the   product for strict associative clans.

\begin{theorem}\label{assoc_strict}
Let $\Xi$ be an associative clan, and suppose that $\tau = (\setc{\hyper{H}_a}{a\in A},\hyper{H})\in \Xi$, $a_0\in A$, and $\tau'= (\setc{\hyper{H}_{(a_0,a')}}{a'\in A'},\hyper{H}_{a_0})\in\Xi$, 
and that we are given constructs $C_a:\hyper{H}_a$ for all $a\in A\backslash \{a_0\}$ 
and constructs $C_{(a_0,a')}:\hyper{H}_{(a_0,a')}$ for all all $a'\in A'$. Taking $\tau''$ to be the grafting of $\tau'$ to  $\tau$ along $a_0$ and  setting $A'':=(A\backslash\{a_0\})\cup\{(a_0,a')\,|\,a'\in A'\}$, denote the corresponding delegations by $\delta''=(\tau'',\{C_{a''}\,|\,a''\in  A'' \})$ and
 $\delta'=(\tau',\{C_{(a_0,a')}\,|\,a'\in A' \})$.
We then have that, for each $\emptyset\neq B''\subseteq A''$, the following {\em polydendriform equation} holds:
$$
\ast^{\tau''}_{B''}(\delta'')=\begin{cases}
\ast^{\tau}_{B''}(\{C_a\,|\,a\in A\backslash\{a_0\}\}\cup\{\ast^{\tau'}(\delta')\}), & \mbox{ if } B''\subseteq A\backslash\{a_0\}\\
\ast^{\tau}_{B}(\{C_a\,|\,a\in A\backslash\{a_0\}\}\cup \{\ast^{\tau'}_{B'}(\delta')\}), & \mbox{ if } B''\not\inc A\backslash\{a_0\}
\end{cases}
,$$ where the superscripts record the respective support teams, and  where, in the second case, $B=(B''\cap(A\backslash\{a_0\}))\cup\{a_0\}$, $B'=\{a'\in A'\,|\, (a_0,a')\in B''\}$ (both non-empty). Moreover, the polydendriform equation  implies the following associativity equation: $$\ast^{\tau''}(\delta'')=\ast^{\tau}(\{C_a\,|\,a\in A\backslash\{a_0\}\}\cup\{\ast^{\tau'}(\delta')\}).$$ 
\end{theorem}

\begin{remark}
The two cases of the polydendriform equation can be drawn as:
\begin{center}

\begin{equation*}
  \left\{
      \begin{aligned}
\begin{tikzpicture}[scale=0.5,grow=up]
\node {$\ast_{B''}$} 
	child{}
	child{}
	child{}
	child{}
	child{}
	child{};
	\draw[blue] (-2,0.5) edge[in=150, out=30](2,0.5);
	\draw[blue] (2,0) node{$A''$};
	\end{tikzpicture}
& = 
\begin{tikzpicture}[scale=0.5,grow=up]
\node {$\ast_{B''}$} 
	child{}
	child{}
	child{node{$\ast$} 
			child{}
			child{}
			child{}
			child{}
	}
	child{}
;
	\draw[blue] (-2,0.2) edge[in=150, out=30](2,0.2);
	\draw[blue] (2,0) node{$A$};

	\draw[blue] (-3,2) edge[in=150, out=30](1,2);
	\draw[blue] (1.2,1.7) node{$A'$};
	\draw[blue] (0,1.1) node{$a_0$};
	
\end{tikzpicture}
\text{if } B''\subseteq A\backslash\{a_0\} \\
\begin{tikzpicture}[scale=0.5,grow=up]
\node {$\ast_{B''}$} 
	child{}
	child{}
	child{}
	child{}
	child{}
	child{};
	\draw[blue] (-2,0.5) edge[in=150, out=30](2,0.5);
	\draw[blue] (2,0) node{$A''$};
	\end{tikzpicture}
& = 
\begin{tikzpicture}[scale=0.5,grow=up]
\node {$\ast_B$} 
	child{}
	child{}
	child{node{$\ast_{B'}$} 
			child{}
			child{}
			child{}
			child{}
	}
	child{}
;
	\draw[blue] (-2,0.2) edge[in=150, out=30](2,0.2);
	\draw[blue] (2,0) node{$A$};

	\draw[blue] (-3,2) edge[in=150, out=30](1,2);
	\draw[blue] (1.2,1.7) node{$A'$};
	\draw[blue] (0,1.1) node{$a_0$};
\end{tikzpicture}
\text{if } B''\not\inc A\backslash\{a_0\}
      \end{aligned}
    \right.
\end{equation*}
\end{center}

\end{remark}

\begin{proof}

We set $\delta=(\tau,\{C_a\,|\,a\in A\backslash\{a_0\}\}\cup\{\ast^{\tau'}(\delta')\})$.
We first show the polydendriform equation.
We proceed by induction on $|H|$. Figure \ref{fig:assoc-cobord} will  help the reader to visualize  the notations  introduced in case (2) of  the  proof.
Denote, for each $a''\in A''$, $X_{a''}:=\mbox{root}(C_{a''})$.  By definition of the operation $\ast_{B''}$, supposing that ${\bf H}, X_{B''}\leadsto \hyper{H}^{B''}_1,\dots ,\hyper{H}^{B''}_{n_{B''}}$, where $X_{B''}=\bigcup_{b''\in B''}X_{b''}$, we have that 
$$\ast_{B''}^{\tau''}(\delta'')= X_{B''}(\ast((\delta'')^{B''}_1),\dots,\ast((\delta'')^{B''}_{n_{B''}})),$$ where, for $1\leq i\leq n_{B''}$, $$(\delta'')_i^{B''}=(\{C_{\widetilde{a''}}:{\bf H}_{\widetilde{a''}}\,|\, \widetilde{a''}\in \widetilde{A''} \mbox{ and } \varphi_{\tau''}^{B''}(\widetilde{a''})=i\},{\bf H}_i^{B''}),$$
with the indexing set
\begin{equation*}
\widetilde{A''}:=A''\backslash B''\cup\{(b'',q)\,|\,b''\in B'' \mbox{ and } 1\leq q\leq n_{b''}\}
\end{equation*}
arising from  ${\bf H}_{b''}, X_{b''}\leadsto \hyper{H}_{(b'',1)},\dots \hyper{H}_{(b'',n_{b''})}$ ($b''\in B''$).   We examine the two cases of the statement in turn.

\begin{figure}
\begin{center}
\includegraphics{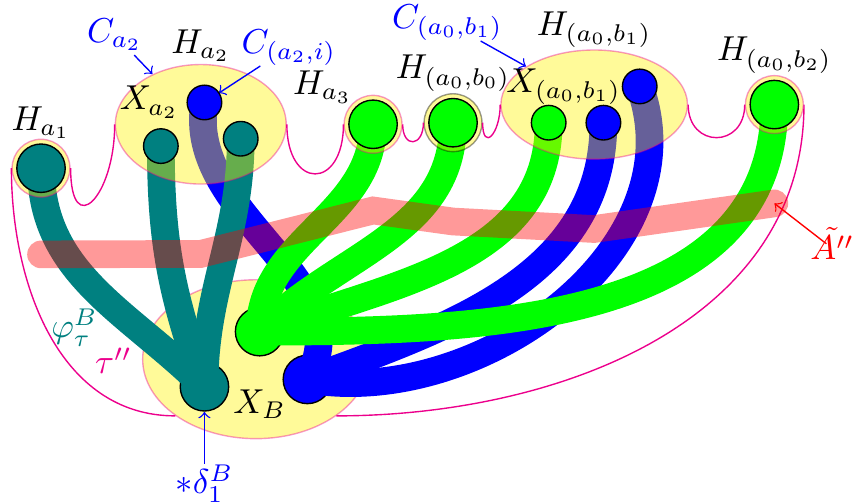} 
\includegraphics{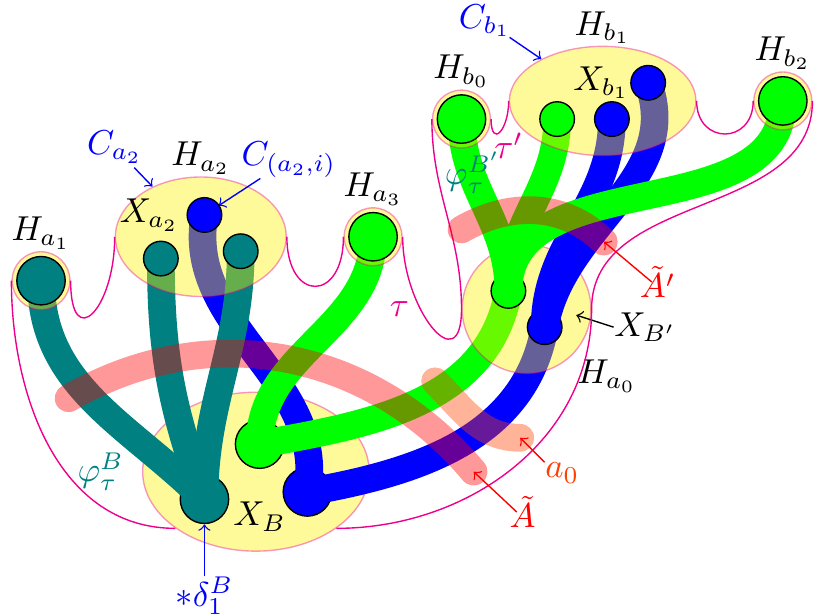}
\end{center}
\caption{Illustration of associativity via cobordisms.
We have: $A=\{a_0, \ldots, a_3\}$, 
$A'=\{b_0, \ldots, b_2\}$,
$A''=\{a_1, \ldots, a_3, (a_0, b_0), (a_0, b_1), (a_0, b_2)\}$,
$B'' = \{a_2,(a_0,b_1)\}$, 
$B' = {b_1}$
and $ B = {a_2,a_0}$.
 \label{fig:assoc-cobord}}
\end{figure}
\begin{enumerate}
\item If $B''\inc A\backslash\set{a_0}$, then, setting ${\bf C}_{a_0}:=\ast^{\tau'}(\delta')$,   we have (using Lemma \ref{linear-construct-in-product}):
$$\ast^{\tau}_{B''}(\delta)= X_{B''}(\ast(\delta^{B''}_1),\dots,\ast(\delta^{B''}_{n_{B''}})),$$
where, for $1\leq l\leq n_{B''}$,
$$\delta^{B''}_l=(\{C_{\tilde{a}}:{\bf H}_{\tilde{a}}\,|\, \tilde{a}\in\tilde{A} \mbox{ and } \varphi_{\tau}^{B''}(\tilde{a})=l\},{\bf H}^{B''}_l),$$
with the indexing set
$$\tilde{A}:=A\backslash
 B''\cup\{(b,p)\,|\, b\in B'' \mbox{ and } 1\leq p\leq n_b\}$$ arising
 from  ${\bf H}_b, X_b\leadsto \hyper{H}_{(b,1)},\dots, \hyper{H}_{(b,n_b)}$ ($b\in B''$).  
Then, establishing  $\ast_{B''}^{\tau''}(\delta'')=\ast_{B''}^{\tau}(\delta)$ amounts to showing  that 
$\ast((\delta'')^{B''}_l)=\ast(\delta^{B''}_l)$, for all $1\leq l\leq n_{B''}$. 

Let $\pi'':\widetilde{A''}\rightarrow A''$ and $\pi:\tilde{A}\rightarrow A$ be the obvious projections (cf. proof of Lemma \ref{Jovana-lemma}). Then it is readily seen (remembering that $H_{(a_0,a')}\inc H_{a_0}$) that $(\pi'')^{-1}(A\backslash\set{a_0})= \widetilde{A''} \cap \tilde{A}= \pi^{-1}(A\backslash\set{a_0})$ and
\begin{equation}
\label{phi-phi-case1}
\varphi_\tau^{B''}|_{\widetilde{A''} \cap \tilde{A}}=\varphi_{\tau''}^{B''}|_{\widetilde{A''} \cap \tilde{A}}\quad \mbox{ and }\quad
\varphi_{\tau}^{B''}\!(a_0)=\varphi_{\tau''}^{B''}(a_0,a'), 
\end{equation}
for all $a'\in A'$. It follows that  for $l\neq \varphi_{\tau}^{B''}(a_0):=l_0$ we have that $(\delta'')^{B''}_l=\delta^{B''}_l$, while (remembering the definition of ${\bf C}_{a_0}$) the equality
$\ast((\delta'')^{B''}_{l_0})=\ast(\delta^{B''}_{l_0})$ follows by induction on $\hyper{H}_{a_0}$.

\smallskip

\item For $B''\not\inc A\backslash\set{a_0}$, let  ${B}:=(B''\inter(A\backslash\set{a_0}))\union \set{a_0}$ and ${B'}:=\setc{a'\in A'}{(a_0,a')\in B''}$.   Let $X_{{B'}}:=\bigcup_{b'\in {B'}} X_{(a_0,b')}$ and suppose that  ${\bf H}_{a_0}, X_{{B'}}\leadsto (\hyper{H}_{a_0})_1^{{B'}},\dots,(\hyper{H}_{a_0})^{{B'}}_{m_{{B'}}}$.  We have by definition
$$\ast_{B'}^{\tau'}(\delta')=X_{{B'}}(\ast((\delta')^{{B'}}_1),\dots,\ast((\delta')^{{B'}}_{m_{{B'}}})),$$
where, for $1\leq j\leq m_{{B'}}$, 
$$(\delta')^{{B'}}_{j}=(\{C_{\widetilde{a'}}:{\bf H}_{\widetilde{a'}}\,|\, \widetilde{a'}\in \widetilde{A'} \mbox{ and } \varphi_{\tau'}^{{B'}}(\widetilde{a'})=j\},({\bf H}_{a_0})^{{B'}}_{j}),$$
with the indexing set 
$$\quad\quad\widetilde{A'}:=\{(a_0,a')\,|  \,a'\in A'\backslash B'\}\} 
\cup \{(a_0,b',k)\,|\, b'\in B' \text{ and } 1\leq k\leq n_{b'}\}$$
arising from  $H_{(a_0,b')}, X_{(a_0,b')}\leadsto \hyper{H}_{(a_0,b',1)},\dots,\hyper{H}_{(a_0,b',n_{b'})}$ ($b'\in B'$). Setting  ${\bf C}^{B'}_{a_0}:=X_{B'}(\ast((\delta')^{B'}_1),\dots,\ast((\delta')^{B'}_{m_{B'}}))$, the equality that we aim to prove displays as  
\begin{equation}\label{eqcont}
\ast_{B''}^{\tau''}(\delta'')=\ast_{B}^{\tau}(\{C_a\,|\, a\in A\backslash\{a_0\}\}\cup\{{\bf C}^{B'}_{a_0}\}).
\end{equation}
Furthermore, by setting $X_{a_0}:=X_{{B'}}$ and $X_{{B}}:= \bigcup_{b\in {B}}X_b$, we can write 
$$X_{B''}=(\bigcup_{b\in {B}\backslash\{a_0\}}X_b)\bigcup \set{X_{{B'}}}= (\bigcup_{b\in {B}\backslash\{a_0\}}X_b)\bigcup \set{X_{a_0}} =
X_{{B}}.$$
We can then  transform \eqref{eqcont}  (applying Lemma \ref{rooted-linear-construct-in-product})
into 
\begin{equation}
X_{B}(\ast((\delta'')^{B''}_1),\dots,\ast((\delta'')^{B''}_{n_{B''}}))=  
X_{{B}}(\ast(\delta^{{B}}_1),\dots,\ast(\delta^{{B}}_{n_{{B}}})),
\end{equation}
where
${\bf H}, X_{{B}}\leadsto \hyper{H}_1^{{B}},\dots,\hyper{H}^{{B}}_{n_{{B}}}$, $n_{{B}}=n_{B''}$, and
for $1\leq l\leq n_{{B}}$, 
$$\delta^B_l=(\{C_{\tilde{a}}:{\bf H}_{\tilde{a}}\,|\, \tilde{a}\in\tilde{A} \mbox{ and } \varphi_{\tau}^{{B}}(\tilde{a})=l\},{\bf H}^{{B}}_l),$$
with the indexing set 
$$\tilde{A}:=A\backslash {B}\cup\{(b,p)\,|\, b\in {B} \mbox{ and } 1\leq p\leq n_b\}$$
arising from ${\bf H}_b, X_b\leadsto \hyper{H}_{(b,1)},\dots,\hyper{H}_{(b,n_b)}$ ($b\in {B}$),
and where 
$$C_{\tilde{a}}=\begin{cases}
C_a, & \mbox{ if } \tilde{a}\in (A\backslash {B}),\\
\ast((\delta')^{{B'}}_{p}), & \mbox{ if } \tilde{a}=(a_0,p),\\
C_{(b,p)}, & \mbox{ if } \tilde{a}=(b,p) \quad(b\in {B}\backslash\set{a_0}).
\end{cases}$$
 Now, since $X_{B''}=X_{\plunderline{B}}$, 
we can suppose, without loss of generality, that $H_{i}^{B''}=H^{\plunderline{B}}_{i}$, for all $1\leq i\leq n_{\plunderline{B}}=n_{B''}$. Therefore, it remains to show that $\ast((\delta'')^{B''}_{i})=\ast(\delta_i^{\plunderline{B}})$. Observe that, since $$\quad\quad\quad \hyper{H}_{(a_0,1)},\dots,\hyper{H}_{(a_0,n_{a_0})}\reflectbox{$\leadsto$}\,{\bf H}_{a_0}\backslash X_{a_0}={\bf H}_{a_0}\backslash X_{\plunderline{B'}}\leadsto (\hyper{H}_{a_0})_1^{\plunderline{B'}},\dots,(\hyper{H}_{a_0})^{\plunderline{B'}}_{m_{\plunderline{B'}}},$$
we have that $n_{a_0}=m_{\plunderline{B'}}$, and  we can assume (without loss of generality) that $H_{(a_0,p)}=(H_{a_0})^{\plunderline{B'}}_p$, for each $1\leq p\leq n_{a_0}$.   
Simple inspection (and standard argumentation with connected components) yields
$$\begin{array}{l}
\pi''^{-1}(A\backslash\set{a_0})= \widetilde{A''}\cap \tilde{A} = \pi^{-1}(A\backslash\set{a_0})\\
\pi''^{-1}(a_0) = \widetilde{A'},
\end{array}$$
where $\pi'':\widetilde{A''}\rightarrow A''$  and $\pi:\tilde{A}\rightarrow A$  are the obvious projections,
and  \begin{equation}\label{phi-phi-case2}
\varphi^{B''}_{\tau''}|_{\widetilde{A''}\cap \tilde{A}}=\varphi_{\tau}^{\plunderline{B}}|_{\widetilde{A''}\cap \tilde{A}} \quad\mbox{ and }\quad  \varphi^{B''}_{\tau''}|_{\widetilde{A'}}=\tilde{\varphi}^{\plunderline{B}}_{\tau}\circ\varphi^{\plunderline{B'}}_{\tau'},\end{equation} 
where   
$\tilde{\varphi}_{\tau}^{\plunderline{B}}(j):={\varphi}_{\tau}^{\plunderline{B}}((a_0,j))$, for each $1\leq j\leq n_{a_{0}}$.
We  note that, thanks  to \eqref{phi-phi-case2}, $\ast((\delta'')^{B''}_{i})$ and $\ast(\delta_i^{\plunderline{B}})$ look respectively like this:
 $$\begin{array}{lll}
\quad \quad\quad \ast((\delta'')^{B''}_{i}) & = &  \ast(\underbrace{\ldots, C_y, \ldots}_{y\in \tilde{A}''\cap \tilde{A},\varphi_{\tau''}^{B''}(y)=i}\,\,\,\,\,\,,\:\ldots\:,\,\underbrace{\ldots ,C_x, \ldots}_{x\in \tilde{A}',\, \varphi^{\plunderline{B'}}_{\tau'}(x)=j\in (\tilde{\varphi}_{\tau}^{\plunderline{B}})^{-1}(i)}  \,\,\,\;,\ldots)\\
\quad\quad\quad\ast(\delta_i^{\plunderline{B}}) &=& \ast(\underbrace{\ldots, C_{y},\ldots}_{y\in \widetilde{A''}\cap \tilde{A},\varphi_{\tau}^{\plunderline{B}}(y)=i}\,\,\,\,\,\,,\:\ldots\:,\,\underbrace{\ast(\ldots, C_x, \ldots)}_{x\in \tilde{A}',\, \varphi^{\plunderline{B'}}_{\tau'}(x)=j\in (\tilde{\varphi}_{\tau}^{\plunderline{B}})^{-1}(i)}  \,\,\,\;,\ldots)
 \end{array}$$  
 and we conclude by applying   induction  to each $\hyper{H}^{B''}_i$ (note that repeated induction, or no induction at all, may be needed for a single fixed $i$, depending on the cardinality of 
 $\varphi^{\plunderline{B'}}_{\tau'}(\tilde{A})
 \cap (\tilde{\varphi}_{\tau}^{\plunderline{B}})^{-1}(i)$).
\end{enumerate}

This concludes the proof of the polydendriform equation.
Associativity is derived as follows. Writing $\delta_{B'}$ for $\{C_a\,|\,a\in A\backslash\{a_0\}\}\cup \{\ast^{\tau'}_{B'}(\delta')\}$, we have  on one hand (in-lining the polydendriform equation):
$$\begin{array}{lllll}
\ast^{\tau''}(\delta'') & = & \overbrace{\sum_{\emptyset\incs B''\subseteq A\backslash\set{a_0}} q^{{|B''|}-1} \:\ast^{\tau}_{B''}(\delta)}^{A_1} & +
& \overbrace{\sum_{\emptyset\incs B''\not\subseteq A\backslash\set{a_0}} q^{{|B''|}-1} \:\ast^{\tau}_B(\delta_{B'})}^{B_1}\\
\end{array}
$$
with $B,B'$ determined from $B''$ as specified in the statement, and on the other hand  (expanding the second summand by linearity):
$$
\ast^{\tau}(\delta)\; =\;  \underbrace{\sum_{\emptyset\incs B\subseteq A\backslash\set{a_0}} q^{{|B|}-1} \:\ast^{\tau}_{B}(\delta)}_{A_2} \; + \;
\underbrace{\sum_{\emptyset\incs B\not\subseteq A\backslash\set{a_0}}  \sum_{\emptyset\incs B'\subseteq A'}q^{|B|+|B'|-2} \:\ast^{\tau}_B(\delta_{B'})}_{B_2}
$$
We have $A_1=A_2$ literally, while $B_1=B_2$ follows by noticing that the map $B''\mapsto ((B''\cap(A\backslash\{a_0\}))\cup\{a_0\},\{a'\in A'\,|\, (a_0,a')\in B''\})$ is bijective. 
\end{proof}

\begin{remark} One could formulate the polydendriform structure as an algebra over a colored operad, where the colors are hypergraphs, the operations are teams, and the carrier of the algebra for the color $\hyper{H}$ is the set of constructs of $\hyper{H}$.
\end{remark}

We shall now relate the polydendriform structure to the tridendriform one, by showing that  the former implies (and can be considered as the unbiased version of) the latter, in the {\em ordered}   framework. 
\smallskip

Let $\Xi$ be an ordered associative clan. Suppose that we have 
$$\{((\hyper{H}_1,\hyper{H}_{2'}),\hyper{H}),   ((\hyper{H}_2,\hyper{H}_{3}),\hyper{H}_{2'}), 
((\hyper{H}_{1'},\hyper{H}_{3}),\hyper{H}),  ((\hyper{H}_1,\hyper{H}_{2}),\hyper{H}_{1'})\}\in\Xi.$$ Denote by $\tau''_1$ the grafting of $((\hyper{H}_1,\hyper{H}_{2}),\hyper{H}_{1'})$ to $((\hyper{H}_{1'},\hyper{H}_{3}),\hyper{H})$ along $1'$, and by $\tau''_2$ the grafting of $((\hyper{H}_2,\hyper{H}_{3}),\hyper{H}_{2'})$ to $((\hyper{H}_1,\hyper{H}_{2'}),\hyper{H})$ along $2'$.  
Note that the above teams are all of the (generic) form
$((\hyper{H}_l,\hyper{H}_r),\hyper{L})$.
We write (cf. Example \ref{permutohedra- })
$$\prec\,:=\ast_{\set{l}} \quad\quad \bcdot\,:=\ast_{\set{l,r}} \quad\quad \succ\,:=\ast_{\set{r}}.$$ 
\begin{proposition} \label{poly-implies-tri}
In the ordered framework, the tridendriform equations follow from the polydendriform ones, relatively to the team $\tau''=((\hyper{H}_1, \hyper{H}_2, \hyper{H}_3), \hyper{H})$.
More precisely, Loday-Ronco's seven equations, as listed in the introduction, correspond to  choosing $B''$ to be $\set{1}$, $\set{2}$, $\set{3}$, $\set{1,2,3}$, $\set{2,3}$, $\set{1,3}$, $\set{1,2}$, respectively.
\end{proposition}
\begin{proof} 
As a sanity check, we first note that there are $2^3-1=7$ non-empty subsets of $\set{1,2,3}$. We check the equation $(\succbcdot)$. Let
$S:\hyper{H}_1$, $T:\hyper{H}_2$, $U:\hyper{H}_3$. We have
{\small
$$
\ast^{\tau''_1}_{\set{2,3}} (S,T,U)  =  \ast_{\set{1’,3}}((\ast_{\set{2}}(S,T),U))=   ((S \succ T) \bcdot U)$$
and $$ \ast^{\tau''_1}_{\set{2,3}} (S,T,U)  = \ast^{\tau''_2}_{\set{2,3}} (S,T,U) =  \ast_{\set{2’}} (S,(\ast_{\set{2,3}}(T,U))) = S \prec (T \bcdot U).$$}
Note that all tridendriform equations follow from the second case of the polydendriform equation, except  $(\precstar)$  and  $(\astsucc)$ (for which we use the first case, and which are the only tridendriform equations involving $\ast$).
\end{proof}

Combining the results of \S\ref{restrictohedra-section} and \S\ref{asp}, we get a whole range of polydendriform/tridendriform structures, and in particular we get structures associated with the graphs  $\bold\Gamma^k$ of Proposition \ref{Gamma-order-friendly}. As we have seen,
for the instances $k=1$ and $k=\infty$ we recover the tridendriform structures of \S\ref{Prologue-section}, thus fulfilling our unifying goal, with a whole infinity of examples sitting ``in the middle''. The case k=2 is that of friezohedra.

\subsection{A non-recursive definition of the product} \label{non-recursive}
In this subsection, we give an equivalent, non-recursive, definition of the   product, directly inspired from \cite{RoncoGTO}. 
 Let $\hyper{H},\hyper{L}$ be two connected hypergraphs such that $H\inc L$ and such that, for all $e\in\hyper{H}$, $e$ is connected in $\hyper{L}$. This entails in particular that $H$ is connected in $\hyper{L}$. Let $S=X(S_1,\ldots,S_n)$  be a construct of $\hyper{L}$, with $S_i:\hyper{L}_i$ where $\hyper{L},X\leadsto\hyper{L}_1,\ldots,\hyper{L}_n$. Then we define a construct $\restrconstr{S}{\hyper{H}}$ of $\hyper{H}$ as follows. We distinguish two cases: 
\begin{itemize}
\item
if $X\cap H=\emptyset$, then there is a unique $j$ such that
$H\inc L_j$, and we set $$\restrconstr{S}{\hyper{H}}=\restrconstr{(S_j)}{\hyper{H}};$$
\item
if $X\cap H\neq\emptyset$, let $\hyper{H},(X\cap H)\leadsto \hyper{H}_1,\ldots,\hyper{H}_p$. This determines a function $\varphi_X^{\hyper{H},\hyper{L}}:\set{1,\ldots,p}\rightarrow\set{1,\ldots,n}$, and we set  $$\restrconstr{S}{\hyper{H}}=(X\cap H)(\ldots,\restrconstr{(S_{\varphi_X^{\hyper{H},\hyper{L}}(i)})}{\hyper{H}_i},\ldots).$$
\end{itemize}
That  $\restrconstr{S}{\hyper{H}}$ is indeed a construct of $\hyper{H}$ is easily seen by induction. 

\begin{example} \label{4.19} In the universe of friezohedra, let us consider $\hyper{H}=\hyper{F}_{\{1, 3, 5\}}$ and $\hyper{L}=\hyper{F}_{\{1, \ldots, 5\}}$. As every edge in $\hyper{H}$ is also in $\hyper{L}$, the hypothesis above is satisfied. Consider the construct $S=3(14(2,5))$ of $\hyper{L}$. Then $\restrconstr{S}{\hyper{H}}=3(1,5)$.
\end{example}

\smallskip
In the next lemma, we give a simpler (but more ``mysterious'')  alternative description of  $\restrconstr{S}{\hyper{H}}$ in terms of tubings. Recall the following notations from \cite{COI}. For every node $Y$ of $S$, we denote by $\valup_S(Y)$  (or simply $\valup(Y)$) the union of the  labels of the descendants of $Y$ in $S$ (all the way to the leaves), including $Y$. By definition of constructs, $\valup_S(Y)$ is always connected in $\hyper{L}$.  We then associate with $S$ the following set of connected subsets, or  {\em tubing}\footnote{We refer to \cite{COI}[Proposition 2] for an exact characterization of inductively defined constructs as tubings. We just note here that the function $\psi$ defined above provides a bijection from constructs to tubings.} (cf.  Footnote \ref{construct-tubing}):
$$\psi(S)= \setc{\valup(Y)}{Y\;\mbox{is a (label of a) node of}\;S}\;.$$
Alternatively, the function $\psi$ is defined recursively by
$$\psi(X(S_1,\ldots,S_n))= \set{L}\Union\: (\Union_{i=1,...,n} \psi(S_i)).$$

\begin{example}
Consider the construct $S$ in Example \ref{4.19}, the associated tubing is:
\begin{equation*}
\psi(S)=\{\{1,2,3,4,5\}, \{1,2,4,5\}, \{5\}, \{2\}\}.
\end{equation*}

\end{example}

We need one definition (adapted to the setting of hypergraphs from \cite{RoncoGTO}). With each $t$ connected in $\hyper{L}$ ($t$ is also called a {\em tube}), we associate a construct $t_{\hyper{H}}$ as follows (note the heterogeneous nature of this definition: we go from tubes to constructs):
\begin{itemize}
\item If $H\inc t$, then we set $t_{\hyper{H}}=H$;
\item if $H\backslash t\neq\emptyset$ yielding $\hyper{H},(H\backslash t)\leadsto \hyper{H}_1,\ldots,\hyper{H}_k$, we set $t_{\hyper{H}}=(H\backslash t)(H_1,\ldots,H_k)$. 
\end{itemize}
This definition can be seen as an instantiation of our definition of $\restrconstr{S}{\hyper{H}}$: more precisely, we can coerce a tube $t$ of $\hyper{L}$ to a construct  $(L\backslash t)(t):\hyper{L}$, and we have 
$t_{\hyper{H}} = \restrconstr{((L\backslash t)(t))}{\hyper{H}}$.

The following lemma asserts that  $\restrconstr{S}{\hyper{H}}$, viewed as a tubing, is entirely determined by the restrictions of the tubes of $S$, thus providing a non-recursive definition for this restriction operation. 

\begin{figure}
\begin{center}
\includegraphics[scale=0.5]{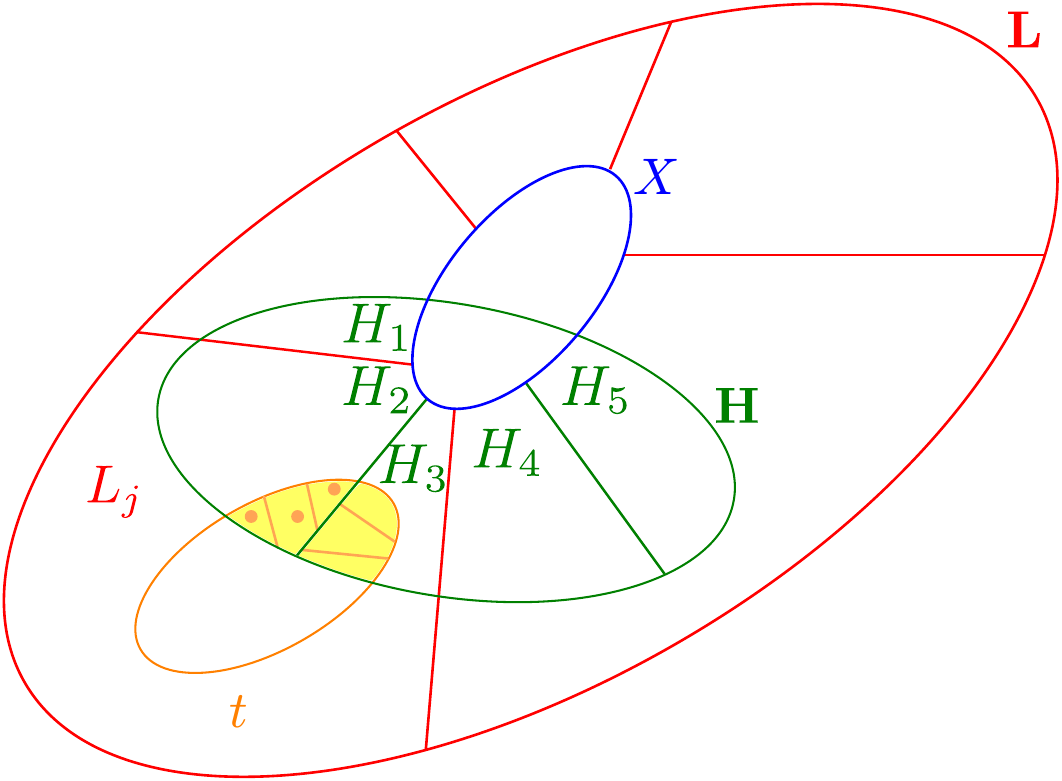} 
\end{center}
{\small \begin{itemize}
\item The compartments with red/blue border are the connected components of $\hyper{L}\backslash X$. 
\item The compartments with green/red/blue border are the connected components of $\hyper{H}\backslash X$.
\item In this example, we have $(\varphi_X^{\hyper{H},\hyper{L}})^{-1}(j)=\set{2,3}$.
\item The small yellow compartments with orange/green borders feature the tubes in $\psi(t_{\hyper{H}})$, 
\item while those additionally marked with a dot are the tubes in $\psi(t_{\hyper{H}_2})$.
\end{itemize}}
\caption{Illustration of the proof of Lemma  \ref{restrconst-as-tubing} \label{restrconst-as-tubing-illust}}
\end{figure}

\begin{lemma} \label{restrconst-as-tubing}
For $\hyper{H},\hyper{L}$ and $S$ as above, we have $\psi(\restrconstr{S}{\hyper{H}})=
\bigcup\setc{\psi(t_{\hyper{H}})}{t\in\psi(S)}$.
\end{lemma}
\begin{proof} (Sketch) Let $S=X(S_1,\ldots,S_n)$ and $L\neq t\in\psi(S)$, i.e., $t\in\psi(S_j)$ for some $j$. Then the statement follows from the  observation (illustrated in Figure \ref{restrconst-as-tubing-illust}) that, with the notation introduced above:
$$\psi(t_{\hyper{H}})= \bigcup\setc{\psi(t_{\hyper{H}_i})}{\varphi_X^{\hyper{H},\hyper{L}}(i)=j} \quad (j,t\in\psi(S_j)\;\mbox{fixed}, i\;\mbox{varying}).$$
Indeed, by definition of $\psi$, we have on one hand that  $(\bigcup\setc{\psi(t_{\hyper{H}})}{t\in\psi(S)})\backslash\set{H}$ is the union of the sets
$(\bigcup\setc{\psi(t_{\hyper{H}})}{t\in\psi(S_j)})$, indexed by $1\leq j\leq n$. On the other hand, applying induction, we have that $\psi(\restrconstr{S}{\hyper{H}})\backslash\set{H}$ is the union of the sets $\psi(t_{\hyper{H}_i})$, for $1\leq i\leq p$ and $t\in \psi(S_{\varphi_X^{\hyper{H},\hyper{L}}(i)})$, which we can repackage as a union indexed by  $j$ (gathering all $i$ such that $\varphi_X^{\hyper{H},\hyper{L}}(i)= j$).  We then conclude by the observation.

\end{proof}
In particular, via the characterization of tubings as constructs, the lemma says that the definition in terms of tubings given in \cite{RoncoGTO}   returns indeed a tubing.

\smallskip
We now come back to the promised alternative definition of the   product. 
Let $\tau=(\{{\bf H}_a\,|\,a\in A\},{\bf H})$ be a team and $U=X(U_1,\ldots,U_n)$ be a construct of $\hyper{H}$. We  associate with $U$ a ``measure''  $\mu^\tau(U)$ as follows (with the notation of \S\ref{tcd}).  We set $B=\setc{b\in A}{X\cap H_b\neq\emptyset}$ and $X_b=X\cap H_b$ for each $b\in B$ (so that $n=n_B$), and we set
$$\mu^\tau(U)=(|B|-1)+ \sum_{1\leq i\leq n_B}\mu^{\tau_i}(U_i).$$The following proposition gives a non-inductive characterization of our product $\ast$.
\begin{proposition}
Let $\delta=(\{C_a:{\bf H}_a\,|\, a\in A\},{\bf H})$ be a delegation of support $\tau$. Then we have:
$$\ast(\delta)= \sum_{U:\hyper{H} \;\mbox{{\small and}}\; \forall a\in A, \restrconstr{U}{\hyper{H}_a}=C_a} q^{\mu^\tau(U)}\, U,$$
and for each $\emptyset\neq B\inc A$, we have that $q^{|B|-1} (\ast_B(\delta))$ is the summand of the above sum where $U$  is further constrained to be such that $\mbox{root}(U)=X_B$.
\end{proposition}
\begin{proof} (Sketch) We use the same notations as above. By unfolding the definition of $U=X(U_1,\ldots,U_n)$, with $X=X_B$, the constraints on $U$ boil down to the constraints (for each $i$) $\restrconstr{(U_i)}{\hyper{H}_{\tilde{a}}}=C_{\tilde{a}}$  for all $\tilde{a}\in\tilde{A}$ such that
$\varphi^{\tau}_B(\tilde{a})=i$. This entails that, taking the right-hand side of the equality and its summands   in the statement as a definition of $\ast$ and $\ast_B$, and noticing that
\begin{equation*}
q^{\mu^\tau(U)}\, X_B(U_1,\ldots,U_n)= q^{|B|-1}\, X_B(\ldots,q^{\mu^{\tau_i}(U_i)}\, U_i,\ldots),
\end{equation*}
 these definitions satisfy the equation $\ast_B(\delta)=(\Union_{b\in B}X_b)(\ast(\delta^B_1),\dots,\ast(\delta^B_{n_{B}})$.
\end{proof}

\begin{example}
We consider the delegation of friezohedra $$\delta=(\{3(1,5)\!:\!\hyper{F}_{\{1,3,5\}}\,,\, 4(2)\!:\!\hyper{F}_{\{2,4\}}\},\hyper{F}_{\{1,2,3,4,5\}}).$$ The   shuffle product of $\delta$ is then given, up to some coefficients, by the sum of all the constructs $U$ of $\hyper{F}_{\{1,\ldots,5\}}$ such that  
$\restrconstr{U}{\hyper{F}_{\{1,3,5\}}} = 3(1,5)$ and $\restrconstr{U}{\hyper{F}_{\{2,4\}}} = 4(2)$. The power of $q$ in the coefficient of $S=3(14(2,5))$ in this sum is given by:
\begin{multline}
\mu^{(\{\hyper{F}_{\{1,3,5\}}, \hyper{F}_{\{2,4\}}\},\hyper{F}_{\{1,2,3,4,5\}})}(S)= (1-1)+ \mu^{(\{\hyper{F}_{\{1\}}, \hyper{F}_{\{5\}}, \hyper{F}_{\{2,4\}}\},\hyper{F}_{\{1,2,4,5\}})}(14(5(2)) \\
= (2-1)+ \mu^{(\{\hyper{F}_{\{2\}}\},\hyper{F}_{\{2\}})}(2) +\mu^{(\{\hyper{F}_{\{5\}}\},\hyper{F}_{\{5\}})}(5)  \\
= 1+ (1-1) + (1-1)=1. \\
\end{multline} 
\end{example}

We note that the non-recursive definition leads to another proof of the polydendriform equation and of associativity -- that is technically more simple but geometrically less appealing than the one we gave in \S\ref{asp} --, based on the observation, say for $(\{{\bf H}_1,{\bf H}_2,{\bf H}_3\},{\bf H})$, $(\{{\bf H}_1,{\bf H}_2\},{\bf H}_{12})$,  $(\{{\bf H}_{12},{\bf H}_3\},{\bf H})$, and $$\delta=(\{S\!:\!{\bf H}_1\,,\, T\!:\!{\bf H}_2\,,\,U\!:\!{\bf H}_3\}\,,\,{\bf H}),$$
 that the data of $V:\hyper{H}$ such that $\restrconstr{V}{\hyper{H}_1}=S$, $\restrconstr{V}{\hyper{H}_2}=T$ and 
 $\restrconstr{V}{\hyper{H}_3}=U$ is equivalent to the data of $V:\hyper{H}$ and $W:\hyper{H}_{12}$ such that $\restrconstr{W}{\hyper{H}_1}=S$, $\restrconstr{W}{\hyper{H}_2}=T$, $\restrconstr{V}{\hyper{H}_{12}}=W$, and $\restrconstr{V}{\hyper{H}_3}=U$\footnote{In turn, this observation relies on the composability of restrictions, i.e., one can prove that 
 $(\restrconstr{V}{\hyper{H}_{12}})\restrconstr{{}}{\hyper{H_1}}= \restrconstr{V}{\hyper{H}_1}$.}.

\subsection{Extending the   framework}\label{semi-strict-subsection} 

In this subsection, we enlarge the coverage of our formalism of teams and clans, and we adapt the   product accordingly, in order to cover other families of polytopes like simplices, hypercubes, or erosohedra. 

\smallskip

A preteam $\tau=(\{{\bf H}_a\,|\,a\in A\},{\bf H})$ is called a {\it semi-strict team} if  for each choice of a   subset $\emptyset\neq B\inc A$ and of a subset $\emptyset\neq X_b\subseteq H_b$ for each $b\in B$, we have that, for each $\tilde{a}\in \tilde{A}$,  

\smallskip
(1) $H_{\tilde{a}}$ is included  in a connected component of $\restrH{H}{(\Union_{b\in B}X_b)}$, or

(2) $|H_{\tilde{a}}|\geq 2$, and, for all  $x\in H_{\tilde{a}}$, $\set{x}$ is a connected component of  $\restrH{H}{(\Union_{b\in B}X_b)}$,

\smallskip
\noindent
where $\tilde{A}$ is as in \S\ref{tcd}.  When (2) applies (and vacuously when $|H_{\tilde{a}}|=1$), we say that $\hyper{H}_{\tilde{a}}$ is {\em dissolved} in $\restrH{H}{(\Union_{b\in B}X_b)}$.
Let us denote with $\tilde{A}_d$ the set of elements $\tilde{a}$ of $\tilde{A}$ such that  case (2) applies.
We  define 
$\overline{A}$ by removing from $\tilde{A}$ all elements $\tilde{a}$ of $\tilde{A}_d$ and replacing them by the elements of $H_{\tilde{a}}$ (thus expressing the atomisation of $\hyper{H}_{\tilde{a}}$), for all $\tilde{a}\in \tilde{A}_{d'}$, i.e., 
$\overline{A}:=(\tilde{A}\backslash \tilde{A}_d) + \union_{\tilde{a}\in \tilde{A}_d} H_{\tilde{a}}$.
The whole situation determines a partition $\overline{A}=\overline{A}_1\union\ldots\union \overline{A}_{n_B}$, and $n_B$ preteams $\tau_i =(\setc{\hyper{H}_{\overline{a}}}{\overline{a}\in \overline{A}_i},\hyper{H}_i)$, where $\hyper{H}_{\overline{a}}$ is defined on the new elements 
$\overline{x}\in  \union_{\tilde{a}\in \tilde{A}_d} H_{\tilde{a}}$
as $\hyper{H}_{\overline{x}}=\set{\set{\overline{x}}}$.
We still use the notation 
$\tau,\Union_{b\in B}X_b\clandec \tau_1,\ldots,\tau_{n_B}$.
The definition of clan is unchanged, except that a clan now consists of semi-strict teams and not of strict teams.
The definition of the    product  is adapted as follows. We assign a  construct $C_{\overline{a}}$ of $\hyper{H}_{\overline{a}}$ for all $\overline{a}\in \overline{A}$, via the following adjustment with respect to the strict case: if $\overline{x}$ is an element of $H_{\tilde{a}}$ for some $\tilde{a}\in \tilde{A}_d$, then we set $C_{\overline{x}}=\set{\overline{x}}$, and we finish as in the strict case: the assignment determines delegations
$\delta_i^B$  $(1\leq i\leq n_B)$, and we define the product exactly as in (\ref{shuffle-def}), but setting $q=-1$ (see below).
 
\smallskip
We can still define a function $\varphi_{\tau}^B$ from $\tilde{A}\backslash\tilde{A}_d$ to $\set{1,\ldots,n_B}$, which we prefer to see as a partial function from $\tilde{A}$ to $\set{1,\ldots,n_B}$. Abusing notation, we can still write (cf. (\ref{deltas}))
$\delta^B_i=(\{C_{\tilde{a}}:{\bf H}_{\tilde{a}}\,|\, \tilde{a}\in \tilde{A} \mbox{ and }\varphi_\tau^B(\tilde{a})=i\},{\bf H}_i^B),$
noticing  that the  participating hypergraphs of  $\tau_i$ that are not the hypergraphs  ${\bf H}_{\tilde{a}}$ with 
$\tilde{a}\in (\varphi_\tau^B)^{-1}(i)$
are all singleton graphs, so that the sloppy notation above extends in a unique way to the ``true'' definition of $\delta^B_i$.
Note however that our abuse of notation is not as innocent as it seems, since the convention relies on the fact that a singleton hypergraph 
$\set{\set{a}}$ admits a unique {\it plain} construct $a$. 
But the same hypergraph admits all $\lambda a$ ($\lambda\in\Bbbk$) as {\it linear} constructs -- a fact that is stressed in the following remark.

\begin{remark} \label{dissolution-irrelevance}
It follows from the definitions that
if $\delta_1$ and $\delta_2$ are delegations of plain constructs having the same support $\tau=(\setc{\hyper{H}_a}{a\in A},\hyper{H})$, if $\delta_1$ and $\delta_2$ differ only on one participating hypergraph $\hyper{H}_{a_0}$,  if $B$ is a non-empty subset of $A$ such that $a_0\nin B$ and $\varphi_\tau^B(a_0)$ is undefined, then $\ast(\delta_1)=\ast(\delta_2)$.  Moreover, if $\delta$ is a (linear) delegation which coincides with $\delta_1$ and $\delta_2$ on all $a\in A\backslash\set{a_0}$ and has in position $a_0$ a linear construct $\sum_{i\in I} \lambda_i C_i$, then we have $\ast(\delta) = (\sum_{i\in I} \lambda_i) \ast\!(\delta_1) \; (= (\sum_{i\in I} \lambda_i) \ast\!(\delta_2))$.
\end{remark}

\smallskip

The notion of associative clan is unchanged. 
The associativity theorem still holds, but only under the assumption $q\!=\!-1$.  The reason for this restriction stems from Remark \ref{dissolution-irrelevance} and from the following lemma.

\begin{lemma} \label{sum-coeff1} If $q\!\!=\!\!-1$, then, for any delegation (in the semi-strict setting) $\delta$, the sum of all coefficients in the expansion of  $\ast(\delta)$ as a linear combination of plain constructs is equal to $1$.
\end{lemma}
\begin{proof} We prove the statement by induction on $|H|$.
From the binomial expansion
 $(1+x)^n=\sum_{0\leq i\leq n}\binom{n}{i} x^i$ expressed as   $(1+x)^n=1 + x(\sum_{1\leq i\leq n}\binom{n}{i} x^{i-1})$
and instantiated with $x=q= -1$, we readily obtain $\sum_{1\leq i\leq n}\binom{n}{i} q^{i-1}=1$.  The statement will then follow if we prove that, for each $\emptyset\subset B\subseteq A$, the sum of the coefficients in the expansion of  
$X_B(\ast(\delta_1),\ldots, \ast(\delta_{n_B}))$ as a linear combination of plain constructs is equal to $1$. But this in turn follows by induction and by multilinearity.
 \end{proof}

\begin{theorem}\label{assoc_semi_strict}
Theorem \ref{assoc_strict} extends to the semi-strict setting for $q=-1$.
\end{theorem}
\begin{proof}
Using the convention above of still defining the   product by appealing to the functions $\varphi_\tau^B$,  the proof of Theorem \ref{assoc_strict} goes through, as long as we do not use the totality of these functions. More precisely, the reasoning in case (1) unfolds without change until the equalities (\ref{phi-phi-case1}) included, which still hold but have now to be understood in the partial sense, i.e., the left-hand side is defined if and only if the right-hand side is defined, in which case they are equal. 

Then two subcases arise.
\begin{enumerate}
\item[(1a)] If $\varphi_\tau^{B''}(a_0)$ is defined, then we conclude case (1) by induction as in the proof of Theorem  \ref{assoc_strict}.
\item[(1b)] Suppose   (new case!) that  $\varphi_\tau^{B''}(a_0)$ is undefined.    Let $\ast^{\tau'}(\delta')=\sum_{i\in I}\lambda_i C_i$.  By Lemma \ref{sum-coeff1}, we have 
$\sum_{i\in I}\lambda_i=1$. Let $\delta'_i$ be the delegation obtained by replacing $\ast(\delta')$ by $C_i$ in $\delta$. By Remark \ref{dissolution-irrelevance}, we have 
$\ast_{B''}(\delta'_i)=\ast_{B''}(\delta'_j)$ for all $i,j$, and, calling  $D$  the common value, we have:
$$\ast_{B''}(\delta)= (\sum_{i\in I}\lambda_i) D = D.$$
On the other hand, by (\ref{phi-phi-case1}), we also have that 
$\varphi_{\tau''}^{B''}(a_0,a')$ is undefined (for all $a'\in A'$), and, again, $\ast^{\tau''}(\delta'')$ does not depend on the  constructs $C_{(a_0,a')}$.  
Moreover, observing that $\delta$ and $\delta''$  coincide on the indices $a\in A\backslash\set{a_0}$,  we get easily that
$\ast_{B''}(\delta'')$ is also equal to the common value $D$, which concludes this new case in the proof of associativity.
\end{enumerate}
Similarly, the reasoning in case (2) unfolds without change until the equalities (\ref{phi-phi-case2}) included, which again hold in the partial sense explained above. Let us repeat here the expressions for $ \ast((\delta'')^{B''}_{i})$ and  for $\ast(\delta_i^{\plunderline{B}})$ that we wrote at this point of the proof of Theorem \ref{assoc_strict}:
 $$\begin{array}{lll}
 \ast((\delta'')^{B''}_{i}) & = &  \ast(\underbrace{\ldots, C_y, \ldots}_{y\in \tilde{A}''\cap \tilde{A},\varphi_{\tau''}^{B''}(y)=i}\,\,\,\,\,\,,\:\ldots\:,\,\underbrace{\ldots, C_x, \ldots}_{x\in \tilde{A}',\, \varphi^{\plunderline{B'}}_{\tau'}(x)=j\in (\tilde{\varphi}_{\tau}^{\plunderline{B}})^{-1}(i)}  \,\,\,\;,\ldots)\\
\ast(\delta_i^{\plunderline{B}}) &=& \ast(\underbrace{\ldots ,C_{y},\ldots}_{y\in \widetilde{A''}\cap \tilde{A},\varphi_{\tau}^{\plunderline{B}}(y)=i}\,\,\,\,\,\,,\:\ldots\:,\,\underbrace{\ast(\ldots ,C_x,\ldots)}_{x\in \tilde{A}',\, \varphi^{\plunderline{B'}}_{\tau'}(x)=j\in (\tilde{\varphi}_{\tau}^{\plunderline{B}})^{-1}(i)}  \,\,\,\;,\ldots)
 \end{array}$$
 The   first expression is still correct, as it displays (with $i$ varying) all elements $y$ and $x$ in the domain of definition $\varphi_{\tau''}^{B''}$, and all constructs involved (the ones appearing explicitly and the ones that have been dissolved) are plain constructs.  The same remarks apply to the second expression, except for the fact that some dissolved constructs are not plain.
 Indeed, we have to look at the  situations $\underbrace{\ldots, C_x ,\ldots}_{x\in \tilde{A}',\, \varphi^{\plunderline{B'}}_{\tau'}(x)=j}$, where 
$\tilde{\varphi}_{\tau}^{\plunderline{B}}(j)$ is undefined.  Then, by (\ref{phi-phi-case2}), we have that  also  $\varphi_{\tau''}^{B''}(x)$ is undefined for all $x\in(\varphi_{\tau'}^{B'})^{-1}(j)$, and 
the corresponding $C_x$  (which are plain, as stressed above) are dissolved  in  $\ast_{B''}(\delta'')$ and hence do not make their way into $(\delta'')^{B''}_{i}$.
On the other hand, the linear constructs  $\ast(\ldots ,C_x,\ldots)$ (where $x$ ranges over $(\varphi_{\tau'}^{B'})^{-1}(j)$ for some $j$ not in the domain of definition of $\varphi_\tau^B$) appear in  $\ast_{B'}(\delta')$, and are also dissolved. 
It follows that the same as what we argued about the first expression  can be argued about  the second one, except for 
the ``trace'' left by the constructs $\ast(\ldots ,C_x,\ldots)$ not being plain constructs, which is taken care of by  reasoning as in case (1b).
Thus also the second expression is  still correct, and the proof of Theorem \ref{assoc_strict}  goes through to the end without change.
\end{proof}
We finish with examples of semi-strict clans that are not strict. 
\begin{example}
The universe formed by all simplices $\hyper{S}^X$ (for a  finite set $X$) gives rise to the semi-strict clan formed by all preteams of the form 
$(\setc{\hyper{S}^{X_a}}{a\in A},\hyper{S}^{\Union X_a})$ (for mutually disjoint $X_a$).  That this clan is not strict is easily checked: given a delegation of constructs $C_a$ and $B\incs A$, all constructs $C_a$ for $a\in A\backslash B$ are dissolved.
The   product instantiates as:

$$
\ast(Y_1(\ldots), Y_2(\ldots), \ldots, Y_n(\ldots)) = \sum_{\emptyset \neq J \subseteq [n]} (\cup_{j \in J} Y_b)(\ldots),
$$
where $(...)$ is a shortcut for a tuple of singletons.
We use this example to illustrate the need to choose $q\!=\!-1$ in the semi-strict setting.
Take $A=\set{a_1,a_2,a_3}$ and $Y_i\inc X_{a_i}$. Then, identifying constructs $Z(\ldots,z,\ldots)$ with their root $Z,$ we have
$\ast_{Y_1}(Y_1,Y_2,Y_3)= Y_1$. 
On the other hand, we have  
$$\ast_{Y_1}(Y_1,\ast(Y_2,Y_3))
 = \ast_{Y_1}(Y_1,Y_2)+ q \ast_{Y_1}(Y_1,Y_2\cup Y_3) + \ast_{Y_1}(Y_1,Y_3)
=  Y_1 +qY_1 +Y_1.$$
Therefore, the two expressions match if and only if $q\!=\!-1$.
\end{example}
\begin{example} \label{hypercube2-example}
 One checks easily that the universe formed by all   hypercubes $\hyper{C}^X$ ($X=\set{x_1<\cdots<x_n}$)  is ordered, and gives rise to the  semi-strict clan formed by all preteams of the form 
$(\{\hyper{C}^{X_1},\ldots,\hyper{C}^{X_n}\},\hyper{C}^{\Union X_i})$, where $\Union_{1\leq i\leq n} X_i$ is endowed with the  order in which  $X_1,\ldots,X_n$ form successive intervals. To illustrate the non-strictness, take the team $(\{\hyper{C}^{\set{x_1<x_2}},\hyper{C}^{\set{x_3<x_4}}\},\hyper{C}^{\set{x_1<x_2<x_3<x_4}})$, and remove $x_1$. Then all of  $\hyper{C}^{\set{x_3<x_4}}$ is dissolved in 
$\hyper{C}^{\set{x_1<x_2<x_3<x_4}}\backslash\set{x_1}=\hyper{S}^{\set{x_2,x_3,x_4}}$.

In the notation introduced at the end of \S\ref{reminders-hypergraph-section}, the tridendriform structure instantiates  as follows ($|v|$ stands for the length of $v$):
$$\begin{array}{llllll}
u \prec v = u \,(-^{|v|}) \\
u \bcdot (v_1 \, +\,  v_2) = u\, (-^{|v_1|})\,\pbullet\, v_2\\
u \succ (v_1 \,+ \, v_2) = \left\{\begin{array}{ll} (u \ast v_1) \,+ \,v_2 & (v_1\neq\epsilon)\\
u+v_2 & (v_1=\epsilon)
\end{array}\right. .
\end{array}$$
\end{example} 
As a last example in this subsection, we describe the $(-1)$-tridendriform products for erosohedra.
\begin{example} \label{eroso2}
Let us first recall that the family of erosohedra is given by:
\begin{align*}
 \hyper{E}^X=\setc{\setc{x_j}{j \neq i}}{1\leq i\leq n},
\end{align*} 
where $X=\{x_1, \ldots, x_n\}$,
and that the constructs of $\hyper{E}^X$ are of two shapes:
\begin{itemize}
\item $x(Y(z_1, \ldots, z_k))$, where $x$ and $z_i$ are singletons of size $1$ in $X$ and $Y$ is a subset of $X$
\item and $Y(z_1, \ldots, z_k)$, where $z_i$ are singletons of size $1$ in $X$ and $Y$ is a subset of $X$ of size at least $2$.
\end{itemize}
 Note that in the first case, $Y(z_1, \ldots, z_k)$ is a construct of a simplex, not of an erosohedron.  Therefore, we take as universe the union of the families of erosohedra and of simplices. Note also that if we order our sets $X$, then we get an ordered universe (and the same is a foritori true for the subuniverse of simplices).
The products on two constructs $S$ and $T$ are given by:
$$\begin{array}{llllll}
(Y(z_1, \ldots, z_k) \prec T) = Y(\ldots) \\
(x(Y(z_1, \ldots, z_k)) \prec T) = x(Y(\ldots)) + x(\mbox{root}(T)(\ldots)) - x((Y \cup \mbox{root}(T))(\ldots))\\\\
(S \succ Y(z_1, \ldots, z_k)) = Y(\ldots) \\
(S \succ x(Y(z_1, \ldots, z_k))) = x(Y(\ldots)) + x(\mbox{root}(S)(\ldots)) - x((Y \cup \mbox{root}(S))(\ldots))\\\\
(S \bcdot T) = (\mbox{root}(S) \cup \mbox{root}(T))(\ldots) 
\end{array}$$
where $(\ldots)$ stands for $(y_1, \ldots, y_p)$ where $\{y_i\}_{i=1}^p$ is the set of elements of $X$ not appearing elsewhere in the construct.
\end{example}
\section{Open questions}

In this section, we list some directions for future work.  We already mentioned the task of  finding a nice combinatorial interpretation of the constructs of friezohedra. Here are some other questions we would like to address.

\smallskip\noindent
$\bullet$  The tridendriform algebras in our examples often satisfy more equations than the tridendrifom ones. Can we make a landscape of the corresponding operad structures?

\smallskip\noindent
$\bullet$ 
  Hopf algebra structures are known for associahedra and permutohedra, see \cite{BR,LR-planar-Hopf}. Can we find sufficient conditions for such structures to exist on a family of polytopes?

\smallskip\noindent
$\bullet$  We would like to explore the ``flip'' order obtained by exchanging the order of elements when seeing constructs as posets, aiming at extending results from \cite{RoncoGTO}. 

\smallskip\noindent
$\bullet$  We also seek comparison results, in the spirit of \cite{LR-planar-Hopf}: given two hypergraphs, one included in the other, what are the relations between the associated polytopes  and between the associated algebras?
\section*{Bestiary of examples}
The examples emphasized in this paper are summed up in the following diagram, where we draw an arrow from $A$ to $B$ if $B$ is ``more truncated'' than $A$, i.e., if the connected subsets of the hypergraph  generating $A$ are connected in the hypergraph generating $B$.
The strict clans are circled.

\begin{center}
\begin{tikzpicture}[scale=0.8]
\node (s) at (0,0) {simplex $\hyper{S}^X$ \cite{Ch00}};
\node (c) at (0,2) {hypercube $\hyper{C}^X$ \cite{Ch00}};
\node[draw, ellipse] (a) at (0,4) {associahedron $\hyper{K}^X$ \cite{LR02}};
\node[draw, ellipse] (pe) at (0,6) {permutohedron $\hyper{P}^X$ \cite{BR}};
\node (e) at (-5,3) {erosohedron $\hyper{E}^X$};
\node[draw, ellipse] (f) at (3,5) {friezohedron $\hyper{F}_X$};
\draw[->] (s)--(c);
\draw[->](c)--(a);
\draw[->](a)--(pe);
\draw[->] (s)--(e);
\draw[->](e)--(pe);
\draw[->] (a)--(f);
\draw[->](f)--(pe);
\end{tikzpicture}
\end{center}

\section*{Glossary}

Below, we sum up the   vocabulary, accompanying each term with an example (in green) coming from the simplices.

A \textcolor{red}{preteam} (like \textcolor{green!50!black}{$(\{\hyper{S}^{1,3}, \hyper{S}^{2,4}\}, \hyper{S}^{1, 2, 3, 4})$}) is the ground for a product. The elements multiplied in the product are constructs of some "participating" hypergraphs and the result is a sum of constructs of the "coordinating hypergraph". A preteam is a pair made of:
\begin{itemize}
\item a set of coordinating hypergraphs,
\item a participating hypergraph.
\end{itemize}
 A \textcolor{red}{strict team} (like \textcolor{green!50!black}{$(\{\hyper{S}^{1,\textcolor{red}{3}}, \hyper{S}^{\textcolor{red}{2,4}}\}, \hyper{S}^{1, 2, 3, 4})$}) is a preteam satisfying some strictness properties (the coordinating hypergraph is more connected than the coordinating hypergraphs in the sense that removing some vertices in some participating  hypergraph disconnects it if the same action disconnects the coordinating hypergraph).

A  \textcolor{red}{strict clan} 
is a set of strict teams satisfying a closure property. It encompasses all the coordinating and participating hypergraphs considered, in any products.

Forming a \textcolor{red}{delegation} consists in picking a strict team and in choosing a construct for every participating hypergraph. These constructs are the objects multiplied in the product.

We finally would like to mention the existence of the terms "polydendriform" \cite{Samuele} and "hypergraphic polytope" \cite{AguiarArdila} 
in the literature, which designate different concepts from the ones presented in this article.

\printbibliography

\end{document}